\documentclass[%
	a4paper, 
	onecolumn, 
	twoside, 
	11pt, 
]{article}

\usepackage[english]{babel}

\usepackage{amsmath}		
\usepackage{amssymb}		
\usepackage{amsthm}			
\usepackage{dsfont}			

\usepackage{graphicx}

\usepackage{pgfplots}
\usetikzlibrary{trees,matrix}

\usepackage[]{todonotes}		
\usepackage{scrextend}

\usepackage[%
	bookmarks,
	raiselinks,
	pageanchor,
	colorlinks,
	citecolor=black,
	linkcolor=black,
	urlcolor=black,
	filecolor=black,
	menucolor=black,
	pagebackref
]{hyperref}

\usepackage[%
	left=3cm,right=2.5cm,top=1.5cm,bottom=1.5cm,
	headheight=7mm,headsep=1.2cm,
	footskip=1.5cm,includeheadfoot
]{geometry}

\usepackage{algorithm}
\usepackage{algpseudocode}

\definecolor{mygreen}{RGB}{0,110,0}
 
\newtheorem{thm}{Theorem}
\newtheorem{defi}{Definition}
\newtheorem{prop}{Proposition}
\newtheorem{lem}{Lemma}
\newtheorem{rem}{Remark}

\newcommand{\argmin}[1]{\underset{{#1}}{\mathrm{argmin}}\;}
\newcommand{\R}[1]{\mathbb{R}^{#1}}

\newcommand{\st}{\qquad\text{s. t.}\qquad}
\newcommand{\norm}[1]{\left\|{#1}\right\|}
\newcommand{\scalpr}[1]{\left\langle{#1}\right\rangle}

\author{Pia Heins\footnote{Westf\"alische Wilhelms-Universit\"at M\"unster,
Institut f\"ur Numerische und Angewandte Mathematik, Einsteinstrasse 62, D 48149
M\"unster, Germany (pia.heins@wwu.de, martin.burger@wwu.de)} \and Michael
Moeller\footnote{Technische Universit\"at M\"unchen, Department of Computer
Science, Informatik 9, Boltzmannstrasse 3, D 85748 Garching, Germany
(m.moeller@gmx.net)} \and Martin Burger\footnotemark[1]}

\title{Locally Sparse Reconstruction Using the
$\ell^{1,\infty}$-Norm}

\begin{document}

\maketitle

\begin{abstract}
This paper discusses the incorporation of local sparsity information,
e.g.
in each pixel of an image, via minimization of the $\ell^{1,\infty}$-norm. We
discuss the basic properties of this norm when used as a regularization
functional and associated optimization problems, for which we derive equivalent
reformulations either more amenable to theory or to numerical computation.
Further focus of the analysis is put on the locally 1-sparse case, which is
well motivated by some biomedical imaging applications.

Our computational approaches are based on alternating direction methods of
multipliers (ADMM) and appropriate splittings with augmented Lagrangians. Those
are tested for a model scenario related to dynamic positron emission
tomography (PET), which is a functional imaging technique in nuclear medicine.

The results of this paper provide insight into the potential
impact of regularization with the $\ell^{1,\infty}$-norm for
local sparsity in appropriate settings. However, it also
indicates several shortcomings, possibly related to the non-tightness of the
functional as a  relaxation of the $\ell^{0,\infty}$-norm.
\end{abstract}

\section{Introduction}

Sparse reconstructions based on minimizing $\ell^1$-norms have gained huge
attention in signal and image processing, inverse problems, and compressed
sensing recently. Their main feature of delivering sparse reconstructions, in
some cases provably the same as with minimizing the nonconvex
$\ell^0$-functional (cf. \cite{Donoho2001, Cohen2009,
Candes2005, Tropp2004, Juditski2010}), is attractive for many applications and has led to
remarkable development in theory and numerics. However, the overall sparsity
enforced by minimal $\ell^1$-norm is not the only kind of prior information
available in practice. Strong recent directions of research are related to
unknowns being matrices, with prior information being e.g. low rank incorporated
via nuclear norm minimization or block sparsity (or collaborative sparsity)
incorporated by minimization of $\ell^{p,1}$-norms with $p \in (1,\infty)$.
Such regularizations have been studied for instance by Fornasier and
Rauhut \cite{Fornasier2008} and Teschke and Ramlau \cite{Teschke2007}
under the name of \emph{joint sparsity}.
Furthermore, details as well as applications of such types of joint sparsity
priors can for instance be found in \cite{Bach2011, Bach2012, Bach2012b}.
Note that there exist several different definitions of mixed $\ell^{p,q}$-norms
in the literature (cf. for instance \cite{Kowalski2009, TroppRelax2006}).
We act in accordance with the definition of mixed norms as proposed in
\cite[Definition 1]{Kowalski2009}.
We refer the reader to \cite{Kowalski2009, Watson1992} and the references
therein for more details on mixed matrix norms.
As a special type of mixed norms we discuss the following type of
sparsity-functional for matrices, namely the $\ell^{1,\infty}$-norm
\begin{align}
	\Vert U \Vert_{1,\infty} = \max_{i\in\{1,\ldots,M\}} \sum_{j=1}^N |u_{ij}|	,
	\label{l1infinity}
\end{align}
where $U\in\R{M\times N}$.
Our motivation is a {\em local sparsity} that frequently appears in inversion
with some spatial dimensions (related to the index $i$) and at least one
additional dimension such as time or spectral information in imaging (related to
the index $j$). The main ingredients in the reconstruction problems we want to
consider are a dictionary $B \in \mathbb{R}^{T\times N}$ (encoding basis
elements), with $N$ basis vectors, $T$ e.g. time steps and $N>T$, a forward
operator $A \in \mathbb{R}^{L\times M}$, with $M$ pixels, $L$ depending on the
application (typically $L<M$ being a number of detectors), and the measured data
$W\in\R{L\times T}$, which yield an inverse problem of the form
\begin{align}
	A U B^T = W. 
\label{basicequation}
\end{align}
The unknown matrix $U \in \mathbb{R}^{M \times N}$ is the collection of the
coefficients with respect to the basis or dictionary encoded in $B$.
Frequently a good dictionary for the local behaviour in the additional
dimension can be found such that the vector $u_{i\bullet}$ is expected to be sparse, which
means we want to minimize
\begin{align*}
	\Vert U \Vert_{0,\infty} = \max_{i\in\{1,\ldots,M\}} \Vert u_{i\bullet} \Vert_0
\end{align*}
with $0^0:=0$ subject to \eqref{basicequation}.
A natural relaxation is to consider the minimization of \eqref{l1infinity}
subject to \eqref{basicequation} instead.

\noindent Examples of applications with such kind of information are:
\begin{itemize}
	\item {\em Dynamic Positron Emission Tomograpy} (PET, and similar problems in
		SPECT, cf. \cite{Wernick2004, AJReader2007, Gunn2002, Gullberg2010}), where
		$i$ refers to the pixel number of the image to be reconstructed and $B$ is a
		local dictionary of (discretized) time-basis functions. The operator $A$ is
		the PET matrix (roughly a sampled Radon-transforms with some corrections) and
		the basis functions are generated by kinetic modeling, i.e. as solutions of
		simple linear ODE systems with unknown parameters. The dictionary is generated
		indirectly by a dictionary for the parameters in the ODE. Clearly one often
		looks for a unique parameter value, i.e. a representation by only one basis
		function in each pixel, an ultimate kind of sparsity.
	\item {\em Fluorescence-lifetime imaging microscopy} (FLIM, cf.
	\cite{Elson2002, Kremers2008})
		where $A$ is a convolution in space or the identity and $B$ contains different
		functions, which characterize the photon decay and are also convolved in time.
		Considering different basis functions for different fluorophores local
		sparsity may enhance the unmixing process.
	\item {\em ECG Cardiac Activation Time Reconstruction} (cf.
		\cite{Oosterom1987, Huiskamp1997}, where $i$ refers to a grid
		point on the epicardial surface and $B$ is a dictionary of step functions
		parametrized by the activation time. Again one looks for a single activation
		time in each grid point, i.e. an ultimately sparse local representation
		usually not formulated this way.
	\item {\em Spectral- and Hyperspectral Imaging} (cf.
	\cite{BioucasDias2012, Esser2012}), where the operator $A$ is often a
		convolution and $B$ is a dictionary of spectral signatures of expected
		elements. Currently resolution is hardly small enough that pixels resolve pure
		materials, but one may easily assume that only very few materials are
		contained in each pixel, which corresponds also to the above local sparsity
		prior.
\end{itemize} 

Our analysis below will demonstrate that it may be advantageous to consider a
combination of minimizing \eqref{l1infinity} with classical $\ell^1$-sparsity.
We will therefore also investigate the more general problem
\begin{equation}
		\min_{U\in\mathbb{R}^{M\times N}} \alpha\lVert
		 U\rVert_{1,\infty} + \beta \lVert
		 U\rVert_{1,1} \quad\text{s.t.}\quad AUB^T=W  \; .
	\label{eq:varmodel0}
\end{equation}
Besides the constrained model  \eqref{eq:varmodel0} we shall also investigate
the unconstrained model
\begin{equation}
		\min_{U\in\mathbb{R}^{M\times N}} \frac{1}{2}\lVert
		\gamma(AUB^{T}-W)\rVert^2_{F}+\alpha\lVert U\rVert_{1,\infty}
		+ \beta \lVert U\rVert_{1,1}   \; ,
	\label{eq:varmodel1}
\end{equation}
which is suited to deal with noisy data. Here we use the Frobenius norm
for the first part. Moreover, $\gamma^T \gamma$ shall be a
positive definite weighting matrix (in a statistical formulation the inverse covariance matrix of the noise). Since
in basically all practical applications one only looks for positive combinations of
basis elements, we shall put a particular emphasis on the case of an
\emph{additional nonnegativity constraint} on $U$ in \eqref{eq:varmodel0},
respectively \eqref{eq:varmodel1}.

We will investigate some basic properties of the model as well as reformulations
with additional inequality constraints. This will make the problem more easily
accessible to detailed analysis and numerical methods. Based on the latter we
also investigate the potential to exactly reconstruct locally 1-sparse signals
by convex optimization techniques.
Before going into the details, we therefore state the following
fundamental definition:
\begin{defi}[$s$-Sparsity]
~\\
	A signal $z$ is called \emph{s-sparse} if it holds
	$$ \norm{z}_0 = s \; .$$
\end{defi}

\subsection{Contributions}

Instead of considering usual sparsity regularizations componentwise on the
unknown matrix such as minimizing the $\ell^{1,1}$-norm, Yuan and Lin
\cite{Yuan2006} considered for instance a generalization of the lasso method (cf.
\cite{Tibshirani1996}), which they call \emph{group lasso}.
The original lasso method (cf. also the review \cite{Tibshirani2011}) is a
shrinkage and selection method for linear regression, i.e. it basically
minimizes the sum of a squared $\ell^2$ data term and an $\ell^1$-regularization
term.
The group lasso, however, generalizes this method by using a slightly different
regularization term, i.e. it minimizes the $\ell^1$-norm of a weighted
$\ell^2$-norm.
Other group lasso generalizations consider the minimization of the $\ell^1$-norm
of the $\ell^{\infty}$-norm, i.e. the $\ell^{\infty,1}$-regularization, cf.
\cite{Vogt2010, Quattoni2009, Schmidt2008}.
Later the $\ell^{2,1}$-regularization was further generalized by Fornasier and
Rauhut \cite{Fornasier2008} and Teschke and Ramlau \cite{Teschke2007}, which
then became known under the term \emph{joint sparsity}. This method mainly
consists of minimizing $\ell^{p,1}$-norms, which are used to include even more
prior knowledge about the unknown such as additional structures like block
sparsity (or collaborative sparsity).

However, for many applications, such as dynamic positron emission tomography or
unmixing problems, it turns out to be useful to incorporate another type of
sparsity. Enhancing the idea of usual $\ell^1$-sparsity to what we call
\emph{local sparsity} is one of the main contributions of this paper.
Local sparsity turns out to be beneficial when working on problems including
inversion with some spatial dimensions and at least one additional dimension
such as time or spectral information.

In order to incorporate the idea of local sparsity, we motivate the use of the
$\ell^{1,\infty}$-norm as regularization functional in a variational
framework for dictionary based reconstruction of matrix completion problems.
Working with the $\ell^{1,\infty}$-norm turns out to be rather difficult, which
is why we additionally propose alternative formulations of the problem.
Besides this, we discuss basic properties of the $\ell^{1,\infty}$-functional
and potential exact recovery.
In addition, we propose a splitting algorithm based on alternating direction
method of multipliers (ADMM) \cite{Glowinski1975, Gabay1976} for the
solution of $\ell^{1,\infty}$-regularized problems and show computational
results for synthetic examples.

\subsection{Organization of this Work}

This paper is organized as follows:

In Section \ref{sec:basic_properties_and_formulations} we will provide the
fundamentals for our method. We will discuss basic
properties of the local sparsity promoting $\ell^{1,\infty}$-regularization such
as different problem formulations, existence and uniqueness. Moreover, we will
give attention to subdifferentials and source conditions for these different
formulations. After showing the equivalence of those formulations, we will
analyze the asymptotics in case that the regularization parameter tends to
infinity.

Section
\ref{sec:scaling_conditions_for_exact_recovery_of_the_constrained_problem}
shall be devoted to the analysis of exact recovery of locally 1-sparse
solutions, where we will introduce certain conditions for exact recovery.

In Section \ref{sec:algorithms} we will propose an algorithm for the
reconstruction with local sparsity, which is based upon the alternating
direction method of multipliers (ADMM) \cite{Glowinski1975, Gabay1976}. 
Computational experiments using this algorithm can be found later on in Section
\ref{sec:numerical_results}.

On the basis of the previous sections, we will apply our model to dynamic positron
emission tomography, which will be used to visualize myocardial perfusion, in
Section \ref{sec:numerical_results}.
We will firstly give a short introduction to the medical and technical
background of dynamic positron emission tomography, before we will briefly
discuss a model for blood flow and tracer exchange, i.e. kinetic modeling.
This will yield the same inverse problem, which we aspire to solve in this
paper.
Afterwards, we will apply the algorithm, which will be deduced in Section
\ref{sec:algorithms}, to artificial data in order to verify our model and
illustrate its potential.

\section{Basic Properties and Formulations}
\label{sec:basic_properties_and_formulations}

In this section we are going to introduce some equivalent formulations of the
main problems \eqref{eq:varmodel0} and \eqref{eq:varmodel1}, which we use for
the analysis later on.
Additionally, we point out some basic properties like convexity, existence and
potential uniqueness.
Furthermore, we propose the subdifferential and discuss a source condition.
Moreover, we prove the equivalence of another reformulation, which
improves the accessibility of the problem for numerical computation.
Finally, we investigate the limit for $\alpha\rightarrow\infty$ and observe what
happens to the optimal solution in that case.

\subsection{Problem Formulations}

Since in most applications a nonnegativity constraint is reasonable, we
firstly restrict \eqref{eq:varmodel0} and \eqref{eq:varmodel1} to this case.
For the sake of simplicity, we define
$$ G:=\left\{U\in\R{M\times N}\;|\;u_{ij}\geq 0 \quad
\forall\;i\in\left\{1,\ldots M\right\},\;j\in\left\{1,\ldots N\right\}\right\}
\; . $$
Hence we have
\begin{align}
	\min\limits_{U\in G} \; \left(\alpha
	\max\limits_{i\in\left\{1,\ldots,M\right\}} \sum\limits_{j=1}^N u_{ij} +
	\beta\sum\limits_{i=1}^M\sum\limits_{j=1}^N u_{ij}\right) \quad\text{s.t.}\quad
	AUB^T=W
	\label{eq:pos_const_prob}
\end{align}
for the constrained problem and
\begin{align}
	\min\limits_{U\in G} \; \left(\frac{1}{2}\left\|
		\gamma(AUB^{T}-W)\right\|^2_{F} + \alpha
		\max\limits_{i\in\left\{1,\ldots,M\right\}} \sum\limits_{j=1}^N u_{ij} +
	\beta\sum\limits_{i=1}^M\sum\limits_{j=1}^N u_{ij}\right)
	\label{eq:pos_unconst_prob}
\end{align}
for the unconstrained problem.

In order to make these problems more easily accessible, we reformulate the
$\ell^{1,\infty}$-term in \eqref{eq:pos_const_prob} and
\eqref{eq:pos_unconst_prob} via a linear constraint:
\begin{thm}[Nonnegative $\ell^{1,\infty}$-Regularization]
~\\
	Let $F\!:\!\mathbb{R}^{M\times N}\longrightarrow \mathbb{R}
	\cup\left\{+\infty\right\}$ be a convex functional. Let
	$\widehat{U}$ be a minimizer of
	\begin{align}
		\min\limits_{U\in G} \; \left(F\left( U \right) + \alpha \max\limits_{i
		\in\{1,\ldots,M\}
		}
		\sum\limits_{j=1}^N u_{ij}\right)
		\label{eq:allg_max_form}
	\end{align}
	and let $(\bar{U},\bar{v})$ minimize
	\begin{align}
		\min\limits_{U\in G,\;v\in \R{+}} \; F\left( U \right) +
		v\quad\text{s.t.}\quad \alpha\sum\limits_{j=1}^N u_{ij}\leq v
		\quad\forall i\in\{1,\ldots,M\}\; .
		\label{eq:allg_U_v_form}
	\end{align}
	Then $\bar{U}$ is a minimizer of \eqref{eq:allg_max_form} and
	$(\widehat{U},\hat{v})$ with
	$\hat{v}=\alpha\max\limits_{i\in\{1,\ldots,M\}}\sum\limits_{j=1}^N\hat{u}_{ij}$
	minimize \eqref{eq:allg_U_v_form}.
\label{th:ell1inf_equiv}
\end{thm}
\begin{proof}
	~\\
	By introducing the constraint $$v=\alpha\max\limits_{i
		\in\{1,\ldots,M\}
		}
		\sum\limits_{j=1}^N
	u_{ij}$$ in problem \eqref{eq:allg_max_form}, we obtain 
	\begin{align*}
		\min\limits_{U\in G}\; F\left(U\right) +
		v\quad\text{s.t.}\quad \alpha \max\limits_{i
		\in\{1,\ldots,M\}}
		\sum\limits_{j=1}^N u_{ij}=v\; ,
	\end{align*}
	which is an equivalent formulation of \eqref{eq:allg_max_form}.
	Now let us consider the inequality-constrained problem
	\eqref{eq:allg_U_v_form}.
	In case that
	$\alpha\max\limits_{i\in\{1,\ldots,M\}}\sum\limits_{j=1}^Nu_{ij}<v$
	holds in \eqref{eq:allg_U_v_form}, the couple $(U,v)$
	cannot minimize \eqref{eq:allg_U_v_form},
	since we can choose $\bar{v}<v$, which is still feasible and reduces the
	objective.
	Thus in the optimal case of problem \eqref{eq:allg_U_v_form} the
	inequality constraint yields the equality constraint,
	i.e.
	$$
	\alpha\max\limits_{i\in\{1,\ldots,M\}}\sum\limits_{j=1}^N\bar{u}_{ij}=\bar{v}$$
	and we see that $\bar{U}$ is a minimizer of \eqref{eq:allg_max_form} and
	$(\widehat{U},\hat{v})$ with
	$\hat{v}=\alpha\max\limits_{i\in\{1,\ldots,M\}}\sum\limits_{j=1}^N\hat{u}_{ij}$
	minimize \eqref{eq:allg_U_v_form}.
\end{proof}

By using Theorem \ref{th:ell1inf_equiv} and defining $F$ as the sum of the
non-negative $\ell^{1,1}$-term and the characteristic function $\chi$
for the data constraint, i.e.
\begin{align*}
	F(U) = \beta\sum\limits_{i=1}^M\sum\limits_{j=1}^N u_{ij} +
	\chi_{\{U\in\R{M\times N}\,|\,AUB^T=W\}}
\end{align*}
with
\begin{align*}
	\chi_{\{U\in\R{M\times N}\,|\,AUB^T=W\}} = 
	\begin{cases}
		0 & \text{if } AUB^T=W, \\
		\infty & \text{else,}
	\end{cases}
\end{align*}
we are able to reformulate problem \eqref{eq:pos_const_prob}
as
\begin{align}
	\min\limits_{U\in G,\;v\in\R{+}} \; \beta\sum\limits_{i=1}^M\sum\limits_{j=1}^N
	u_{ij} + \chi_{\{U\in\R{M\times N}\,|\,AUB^T=W\}}
	+ v \quad \text{ s.t. } \quad \alpha\sum\limits_{j=1}^N u_{ij}\leq v \; .
	\label{eq:Uv_constr_prob}
\end{align} 
Likewise we obtain the unconstrained problem from \eqref{eq:pos_unconst_prob} as
\begin{align}
	\min\limits_{U\in G,\;v\in\R{+}} \;
	\frac{1}{2}\left\|\gamma\left(AUB^T-W\right)\right\|_F^2 +
	\beta\sum\limits_{i=1}^M\sum\limits_{j=1}^N u_{ij} + v \quad \text{ s.t. }
	\quad \alpha\sum\limits_{j=1}^N u_{ij}\leq v
	\label{eq:Uv_unconstr_prob}
\end{align}
by using the sum of the data fidelity and the non-negative $\ell^{1,1}$-term as
functional $F$.

In order to understand the potential exactness of sparse reconstructions, we
will focus on the analysis of \eqref{eq:Uv_constr_prob} in Section
\ref{sec:scaling_conditions_for_exact_recovery_of_the_constrained_problem},
however, \eqref{eq:Uv_unconstr_prob} is clearly more useful in practical
situations when the data are not exact. Thus it builds the basis for most of the
further analysis and in particular for computational investigations.

However, we firstly propose another formulation, which shall make the problem
more easily accessible for the numerical solution.
For this reformulation we show in Subsection \ref{equivalence_of_formulations}
that $\max\limits_{i\in\{1,\ldots,M\}}\sum\limits_{j=1}^N u_{ij}(\alpha)$
depends continuously on the regularization parameter $\alpha$ in problem
\eqref{eq:allg_max_form}.
Then we prove that $\alpha\mapsto\max\limits_{i\in\{1,\ldots,M\}}\sum\limits_{j=1}^N
u_{ij}(\alpha)$ is monotonically decreasing and finally we analyze its limits
for $\alpha$ going to zero and infinity. We can then show that under certain
circumstances the support of the minimizers of \eqref{eq:allg_max_form} and
\begin{align}
	\min_{U\in G} \; F\left(U\right)\quad\text{ s.t. }\quad \sum\limits_{j=1}^N
	u_{ij}\leq\tilde{v}
\label{eq:easy_problem}
\end{align}
coincide for a certain fixed $\tilde{v}$.
Thus instead of regularizing with $\alpha$ we can now use $\tilde{v}$ as a
regularization parameter.

\subsection{Existence and Uniqueness}

Let us now discuss some basic properties of $\ell^{1,\infty}$-regularized
variational problems.
We show that there exists a minimizer for these problems and discuss potential
uniqueness.

\subsubsection{Existence}

Since the $\ell^{1,\infty}$-regularization functional is a norm it is also
convex.
This holds for the non-negative formulation \eqref{eq:nonneg_func} as well,
which directly follows from the fact that we add the characteristic function of
a convex set to the convex norm functional:

\begin{prop}[Convexity of the Non-Negative $\ell^{1,\infty}$-Functional]
\label{th:convexity_pos_ell1infty}
~\\
	Let be $U\in\R{M\times N}$. The functional
	\begin{equation}
		\mathcal{R}\!\left(U\right) := 
		\begin{cases}
			\max\limits_{i\in\{1,\ldots,M\}}\sum\limits_{j=1}^N u_{ij} & \text{if }\; u_{ij}\geq 0, \\
			\infty & \text{else,}
		\end{cases}
	\label{eq:nonneg_func}
	\end{equation}
	 is convex.
\end{prop}

\begin{prop}[Lower Semi-Continuity]
\label{th:lower_semi-continuity}
~\\
	Let $F\!:\!\mathbb{R}^{M\times N}\longrightarrow \mathbb{R}
	\cup\left\{+\infty\right\}$ be a convex functional.
	Then \eqref{eq:allg_max_form} and \eqref{eq:allg_U_v_form} are lower
	semi-continuous.
\end{prop}
\begin{proof}
~\\
	Since $F(U)$ is convex, we can conclude that \eqref{eq:allg_max_form} is convex
	by using Proposition \ref{th:convexity_pos_ell1infty}. Due to Theorem
	\ref{th:ell1inf_equiv} we deduce that \eqref{eq:allg_U_v_form} is convex as well.
	
	Our problem is finite dimensional and hence all norms are equivalent. In
	addition it contains only linear inequalities. Therefore, we can deduce lower
	semi-continuity directly from convexity, which we already have. 
\end{proof}
Let us now analyze the existence of minimizers of the different
problems.
\begin{thm}[Existence of a Minimizer of the Constrained
Problem]
\label{th:existence_constr_prob}
~\\
	Let there be at least one $\widetilde{U}\in G$ that satisfies
	$A\widetilde{U}B^T=W$.
	Then there exists a minimizer of the constrained problems
	\eqref{eq:pos_const_prob} and \eqref{eq:Uv_constr_prob}.
\end{thm}
\begin{proof}
~\\
	Since we only have linear parts, we see that
	\begin{align*}
		F(U) := \beta\sum\limits_{i=1}^M\sum\limits_{j=1}^N u_{ij} + 
		\chi_{\{U\in\R{M\times N}\,|\,AUB^T=W\}}
	\end{align*}
	is convex.
	Then Proposition \ref{th:lower_semi-continuity} leads to lower semi-continuity
	of \eqref{eq:pos_const_prob} and \eqref{eq:Uv_constr_prob}.
	
	We still need to show that there exists a $\xi$ such that the sublevel set
	$$	\mathcal{S}_{\xi} = \left\{U\in G \; \left| \;
		\beta\sum\limits_{i=1}^M\sum\limits_{j=1}^N u_{ij} +
		\alpha\max\limits_{i\in\{1,\ldots,M\}}\sum\limits_{j=1}^N u_{ij} + 
		\chi_{\{U\in\R{M\times N}\,|\,AUB^T=W\}}
		\leq
		\xi \right.\right\} $$
	is compact and not empty.
	
	With $\widetilde{U}$ we have a feasible element and we can define
	$$ \xi := \beta\sum\limits_{i=1}^M\sum\limits_{j=1}^N \tilde{u}_{ij} +
		\alpha\max\limits_{i\in\{1,\ldots,M\}}\sum\limits_{j=1}^N \tilde{u}_{ij} \; .$$
	Due to the fact that $\widetilde{U}\in \mathcal{S}_{\xi}$
	holds, we see that $\mathcal{S}_{\xi}$ is not empty and since for all
	$U\in \mathcal{S}_{\xi}$ holds that
	\begin{align*}
		\norm{U}_{1,\infty} \leq \frac{\xi}{\alpha} \quad\text{if}\quad \alpha\neq
		0 \qquad\text{or}\qquad
		\norm{U}_{1,1} \leq \frac{\xi}{\beta} \quad\text{if}\quad \beta\neq 0\
	\end{align*}
	and we have $u_{ij}\geq 0$ for all $i\in\{1,\ldots,M\}$ and
	$j\in\{1,\ldots,N\}$, the sublevel set $\mathcal{S}_{\xi}$ is bounded.

	Our functional is finite dimensional, hence $\mathcal{S}_{\xi}$ is bounded in
	all norms. Furthermore, boundedness of $\mathcal{S}_{\xi}$ in combination with
	lower semi-continuity of \eqref{eq:pos_const_prob} and
	\eqref{eq:Uv_constr_prob} yields compactness of $\mathcal{S}_{\xi}$. Finally,
	we obtain the existence of a minimizer of the constrained problems
	\eqref{eq:pos_const_prob} and \eqref{eq:Uv_constr_prob}.
	Note that we choose the minimizing $v$ for problem \eqref{eq:Uv_constr_prob} in
	accordance with Theorem \ref{th:ell1inf_equiv}.
\end{proof}

\begin{thm}[Existence of a Minimizer of the Unconstrained
Problem]
\label{th:existence_unconstr_prob}
~\\
	Let be $\alpha>0$. Then there exists a minimizer of \eqref{eq:pos_unconst_prob}
	and \eqref{eq:Uv_unconstr_prob}.
\end{thm}
\begin{proof}
~\\
	Obviously the functional
	\begin{align*}
		F(U) := \frac{1}{2}\left\|\gamma\left(AUB^T-W\right)\right\|_F^2 +
		\beta\sum\limits_{i=1}^M\sum\limits_{j=1}^N u_{ij}
	\end{align*}
	is convex.
	By using Proposition \ref{th:lower_semi-continuity} we obtain that
	\eqref{eq:pos_unconst_prob} and \eqref{eq:Uv_unconstr_prob} are lower
	semi-continuous.\\
	The sublevel set
	$$	\mathcal{S}_{\xi} = \left\{U\in G \; \left| \;
		\frac{1}{2}\left\|\gamma\left(AUB^T-W\right)\right\|_F^2 +
		\beta\sum\limits_{i=1}^M\sum\limits_{j=1}^N u_{ij} +
		\alpha\max\limits_{i\in\{1,\ldots,M\}}\sum\limits_{j=1}^N u_{ij} \leq \xi
		\right.\right\} $$
	with
	$$ \xi := \frac{1}{2}\left\|\gamma W\right\|_F^2 $$
	is not empty, since we obviously have $0\in \mathcal{S}_{\xi}$.\\
	Analogously to the proof of Theorem \ref{th:existence_constr_prob}, we see that
	$\mathcal{S}_{\xi}$ is bounded.
	Due to the finite dimensionality of the problem, we have compactness of the
	sublevel set $\mathcal{S}_{\xi}$.
 	Together with semi-continuity we obtain existence of a minimizer of the
	unconstrained problems \eqref{eq:pos_unconst_prob} and
	\eqref{eq:Uv_unconstr_prob}.
	Note that we choose again the minimizing $v$ for problem
	\eqref{eq:Uv_constr_prob} in accordance with Theorem \ref{th:ell1inf_equiv}.
\end{proof}

\subsubsection{Uniqueness}

Let us now shortly discuss potential uniqueness of the solutions of
\eqref{eq:Uv_constr_prob} and \eqref{eq:Uv_unconstr_prob}.

\begin{thm}[Restriction of the Solution Set]
\label{th:minimizer_with_minimal_v}
~\\
	There exists a solution $\left(\bar{U},\;\bar{v}\right)$ of
	\eqref{eq:Uv_constr_prob} and \eqref{eq:Uv_unconstr_prob} with $\bar{v}$
	minimal, i.e. $\bar{v}\leq v$ for all minimizers $\left(U,\;v\right)$.
	Furthermore $\bar{v}$ is unique.
\end{thm}
\begin{proof}
~\\
	Obviously $\bar{v}$ can be defined as
	\begin{align*}
		\bar{v} := \inf\left\{v\;|\;(U,\;v)\text{ is a minimizer of }
		\eqref{eq:allg_U_v_form}\right\}
	\end{align*}
	with $F$ as in \eqref{eq:Uv_constr_prob}, \eqref{eq:Uv_unconstr_prob}
	respectively.
	Due to Theorem \ref{th:existence_constr_prob} and
	\ref{th:existence_unconstr_prob}, we know that $\bar{v}<\infty$ has to hold.
	We proof the assumption via contradiction.\\
	Assume there does not exist a $\bar{U}$ with
	$(\bar{U},\;\bar{v})$ being a minimizer of \eqref{eq:Uv_constr_prob}, \eqref{eq:Uv_unconstr_prob}
	respectively.
	We can find a sequence of minimizers $(U_k,\;v_k)$ with
	$v_k\rightarrow\bar{v}$.
	$U_k$ is bounded, since $U_k\in \mathcal{S}_{\xi}$ holds for all $k$.
	Thus there exists a converging subsequence $(U_{k_l},\;v_{k_l})$.
	Finally, lower semi-continuity provides us with the limit $(\bar{U},\;\bar{v})$
	being a minimizer, which is a contradiction to the assumption.
	Furthermore, since we have $\bar{v}\in\R{+}$, it is obviously unique.
\end{proof}

In Theorem \ref{th:minimizer_with_minimal_v} we have seen that we can reduce
the solution set to those solutions with optimal $v$, i.e. 
\begin{align*}
	\mathcal{S} := \left\{(\bar{U},\;\bar{v}) \in G \times\R{+} \text{ is a minimizer of }
			\eqref{eq:allg_U_v_form} \;|\; \bar{v}
	\;\text{minimal} \right\} \; ,
\end{align*} 
with $F$ as in \eqref{eq:Uv_constr_prob}, \eqref{eq:Uv_unconstr_prob}
respectively.
There always exists a unique $\bar{v}\in\R{+}$, however, in
general we are not able to deduce uniqueness for $(U,\;\bar{v}) \in G
\times\R{+}$.

\subsection{The Subdifferential of the
\texorpdfstring{$\ell^{1,\infty}$}{l1inf}-Norm for Matrices}
\label{subdifferentials_source_condition}

In this subsection we characterize the subdifferentials of the
$\ell^{1,\infty}$-norm and its nonnegative counterpart.
Furthermore, we discuss what kind of solutions $\widehat{U}$ to
$A\widehat{U}B^T=W$ are likely to meet a source condition for the
$\ell^{1,\infty}$-regularization.

\subsubsection{The Subdifferential of
\texorpdfstring{$\ell^{1,\infty}$}{ell1inf}}

We start by computing the subdifferential of the $\ell^{1,\infty}$-norm.
Note that while in general the subdifferential of a convex
function $J$ depending on a matrix $U\in\R{M\times N}$ is defined as
\begin{align*}
	\partial J(U) = \left\{P \in\R{M\times N} ~\big|~ J(V) - J(U) - \langle P, V-U
	\rangle_F \geq 0, \ \forall\, V \in\R{M\times N} \right\}
\end{align*}
with $\langle A, B \rangle_F = \sum_{i,j} A_{i,j}B_{i,j}$, one readily shows
that the subdifferential of an absolutely 1-homogeneous convex functional may
also be characterized as
\begin{align}
	\label{eq:charac1homSubdiff}
	\partial J(U) = \left\{P \in\R{M\times N} ~\big|~ J(U) = \langle P, U \rangle_F
	, \ J(V) \geq \langle P, V \rangle_F , \ \forall\, V \in\R{M\times N} \right\}.
\end{align}
\begin{thm}[Subdifferential of the $\ell^{1,\infty}$-Functional]
\label{thm:subdiff}
~\\
	Let be $U,\; P\in\R{M\times N}$. The subdifferential of 
	$\|U\|_{1,\infty}$ can be characterized as follows:\\
	Let $I$ be the set of indices, where $U$ attains its maximum row-$\ell^1$-norm,
	i.e.
	\begin{align*}
		I = \left\{ i \in \{1,\ldots, M \} ~\left|~ \sum_{j=1}^N |u_{ij}| = \max_{m\in
		\{1,\ldots, M \}} \sum_{j=1}^N |u_{mj}| \right.\right\}.
	\end{align*}
	Then the following equivalence holds:
	\begin{align}
		P \in \partial \|U\|_{1,\infty} \quad\Leftrightarrow\quad 
		\begin{cases}
			p_{ij} = \omega_{i}\,\text{sign}(u_{ij}) & \text{if } i \in I,\\
			p_{ij} = 0  & \text{if } i \notin I,
		\end{cases}
	\label{eq:charac}
	\end{align} 
	with weights $\omega_i \geq 0$ such that $\sum\limits_{i\in I} \omega_i = 1$ holds if
	$U \not\equiv 0$ and $\sum\limits_{i\in I} \omega_i \leq 1$ holds
	if $U \equiv 0$.
	By convention we use $\text{sign}(0)$ to denote an arbitrary element in
	$[-1,1]$.
\end{thm}
\begin{proof}
~\\
	First, assume that a given $P$ meets the conditions on the right hand side of
	\eqref{eq:charac}. We have
	\begin{align*}
		\sum\limits_{i=1}^M\sum\limits_{j=1}^N p_{ij}u_{ij} &= \sum_{i\in I}
		\sum_{j=1}^N \omega_i\; \text{sign}(u_{ij})u_{ij} = \sum_{i\in I} \omega_i
		\sum_{j=1}^N |u_{ij}| = \|U\|_{1,\infty} \sum_{i\in I} \omega_i = 
		\|U\|_{1,\infty}
	\end{align*} 
	and
	\begin{align*}
		\sum\limits_{i=1}^M\sum\limits_{j=1}^N p_{ij}v_{ij} & \leq  \left|
		\sum\limits_{i=1}^M\sum\limits_{j=1}^N p_{ij}v_{ij} \right| \leq \sum_{i \in
		I}\sum_{j=1}^N \left|p_{ij}v_{ij}\right| \leq \sum_{i \in I} \omega_i
		\sum_{j=1}^N |v_{ij}| \\
		& \leq \sum_{i \in I} \omega_i \max_m\sum_{j=1}^N |v_{mj}| \leq
		\max_m\sum_{j=1}^N |v_{mj}| = \|V\|_{1,\infty} \; .
	\end{align*}
	By the characterization of the subdifferential \eqref{eq:charac1homSubdiff} we
	obtain $P \in \partial \norm{U}_{1,\infty}$.

	Now let $P \in \partial \norm{U}_{1,\infty}$ be given. Note that the considered
	$\ell^{1,\infty}$ matrix norm, is also the operator norm induced by the
	$\ell^\infty$ vector norm, i.e.
	$$ \|U\|_{1,\infty} = \max_{\|v\|_\infty \leq 1} \|Uv\|_{\infty}. $$
	The latter allows us to apply Theorem 4 of Watson in \cite{Watson1992} and
	conclude that
	\begin{align*}
		\partial \|U\|_{1,\infty} &= \text{conv}\left\{wv^T~\big|~\|v\|_\infty = 1,  \
		Uv=\|U\|_{1,\infty} z,\  \|z\|_\infty = 1, \ w \in \partial \|z\|_\infty
		\right\}
	\end{align*}
	holds, where $\text{conv}\{\cdot\}$ shall denote the convex hull.
	We will first show that every $P = wv^T$ from the above set can be written as
	the claimed right hand side in \eqref{eq:charac} and conclude by noting that
	the right hand side in \eqref{eq:charac} corresponds to a convex set. For $P =
	wv^T$ the above conditions imply that there is at least one $i \in I$ such that
	$v_j = \text{sign}(u_{ij})$ or $v_j = -\text{sign}(u_{ij})$ is true. For
	every $i$ where the above holds, we have either $z_i = 1$ or $z_i = -1$ (unless
	we have $U\equiv 0$, which we will consider later).
	Now $w \in \partial \|z\|_\infty$ means that we have $w_i = 0$ in case that
	$|z_i|<1$ holds and $w_i = z_i \omega_i$ with $\omega_i \geq 0$ and $\sum_{i}
	\omega_i = 1$, else. 
	Noting that the ambiguity in the sign of $v_j$ cancels after the multiplication
	with $w_i = z_i \omega_i$ yields the equivalence to our characterization in the
	case of $U \not \equiv 0$. Otherwise $v$ and $u$ are arbitrary, which
	particularly means that $w$ is any element in $\partial \|0\|_\infty$, i.e.
	$\|w\|_1\leq 1$. Thus, the subdifferential becomes the set of all $wv^T$ such
	that $\|w\|_1\leq 1$ and $\|v\|_\infty\leq 1$ hold, which yields our second
	assertion since we have $I =  \{1,\ldots, M \}$.
	
	Finally, note that for $P_1$ and $P_2$ both meeting the right hand side of our
	claimed characterization of the subdifferential (i.e. meeting Watson's
	conditions without  $\text{conv}\{\cdot\}$) we find that $\alpha P_1 + (1-\alpha P_2)$,
	$\alpha \in [0,1]$, again meets the conditions of our right hand side, such
	that we can conclude the convexity of the set.
\end{proof}
One particular thing we can see from Theorem \ref{thm:subdiff} is
that $P \in \partial \|U\|_{1, \infty}$ for an arbitrary $U$ meets 
$$\norm{P}_{\infty,1} = \sum_{i=1}^M\max_{j\in\{1,\ldots,N\}} |p_{ij}| \leq 1$$
and $$\norm{P}_{\infty,1} = \sum_{i=1}^M \max_{j\in\{1,\ldots,N\}} |p_{ij}| = 1
\qquad\text{for}\qquad U\not\equiv 0 \; .$$ 
Thus we see that the $\ell^{\infty,1}$-norm is the dual to the
$\ell^{1,\infty}$-norm, which has already been observed by Tropp in
\cite{TroppRelax2006}.

\subsubsection{The Subdifferential of the Nonnegative
\texorpdfstring{$\ell^{1,\infty}$}{ell1inf}-Formulation}

Let us now consider the nonnegative $\ell^{1,\infty}$-functional
\eqref{eq:nonneg_func}.
\begin{thm}[Subdifferential of the Nonnegative $\ell^{1,\infty}$-Functional]
~\\
	Let $P^{1,\infty}\in\R{M\times N}$ be the subdifferential of the
	$\ell^{1,\infty}$-norm characterized as before in \eqref{eq:charac} and let
	$p_{ij}^{1,\infty}$ be its entries for $i\in\{1,\ldots,M\}$ and
	$j\in\{1,\ldots,N\}$.
	Then the subdifferential of the nonnegative $\ell^{1,\infty}$-functional
	\begin{align*}
		\mathcal{R}\!\left(U\right) := 
		\begin{cases}
			\max\limits_{i\in\{1,\ldots,M\}}\sum\limits_{j=1}^N u_{ij} & \text{if } u_{ij}\geq 0, \\
			\infty & \text{else,}
		\end{cases}
	\end{align*}
	can be characterized as
	\begin{align*}
		P\in\partial \mathcal{R}\!\left(U\right) \quad\Leftrightarrow\quad p_{ij} =
		p_{ij}^{1,\infty} + \mu_{ij} \; ,
	\end{align*}
	where $\mu_{ij}$ for all $i\in\{1,\ldots,M\},\;j\in\{1,\ldots,N\}$ are the
	Lagrange parameters with
	\begin{align*}
		\mu_{ij}
		\begin{cases}
			= 0 & \text{if } u_{ij}\neq 0, \\
			\leq 0 & \text{if } u_{ij} = 0.
		\end{cases}
	\end{align*}
\end{thm}
\begin{proof}
~\\
	$\mathcal{R}\!\left(U\right)$ can be written using the characteristic function,
	i.e.
	\begin{align*}
		\mathcal{R}\!\left(U\right) = \norm{U}_{1,\infty} + \chi_{\{U\in\R{M\times
		N}\,|\,u_{ij}\,\geq\, 0\;\;\forall
		\;i\in\{1,\ldots,M\},\;j\in\{1,\ldots,N\}\}} \; .
	\end{align*}
	In case that the subdifferential is additive, we have
	\begin{align*}
		\partial \mathcal{R}\!\left(U\right) = \partial\norm{U}_{1,\infty} + \partial\chi_{\{U\in\R{M\times
		N}\,|\,u_{ij}\,\geq\,0\;\;\forall
		\;i\in\{1,\ldots,M\},\;j\in\{1,\ldots,N\}\}}
	\end{align*}
	and directly obtain
	\begin{align*}
		P\in\partial \mathcal{R}\!\left(U\right) \quad\Leftrightarrow\quad p_{ij} = p_{ij}^{1,\infty} +
		\mu_{ij} \; .
	\end{align*}
	In order to prove that in this case the subdifferential is additive, we have to
	show that the following two conditions hold (cf. \cite[Chapter 1,
	Proposition 5.6]{Ekeland1999}):
	\begin{enumerate}
	 	\item $\norm{U}_{1,\infty}$ and $\chi_{\{U\in\R{M\times N}\,|\,u_{ij}\geq
	 	0\;\;\forall \;i,\;j\}}$ are proper, convex and lower semi-continuous,
	 	\item  there exists a
	 	$\bar{U}\in\text{dom}\norm{U}_{1,\infty}\cap\text{dom}\chi_{\{U\in\R{M\times
	 	N}\,|\,u_{ij}\geq 0\;\;\forall \;i,\;j\}}$, where one of the two functionals
	 	is continuous.
	\end{enumerate}
	This is quite easy to see:
	\begin{enumerate}
	 	\item Since $\norm{U}_{1,\infty}$ is a norm, it is naturally proper and convex. Furthermore, we see that
	 	$\chi_{\{U\in\R{M\times N}\,|\,u_{ij}\geq 0\;\;\forall
	 	\;i\in\{1,\ldots,M\},\;j\in\{1,\ldots,N\}\}}$ is obviously
	 	proper. It is also convex, since the characteristic
	 	function of a convex set is also convex.
		Both functionals are lower semi-continuous, since we are in a finite
		dimensional setting.
		\item Let $\bar{u}_{ij}>0$ hold for all $i\in\{1,\ldots,M\}$ and
		$j\in\{1,\ldots,N\}$.
		Then we have $$
		\bar{U}\in\text{dom}\norm{U}_{1,\infty}\cap\text{dom}\chi_{\{U\in\R{M\times
		N}\,|\,u_{ij}\geq 0\;\;\forall \;i,\;j\}} \; .$$ Furthermore, both functionals
		are continuous at $\bar{U}$.
	\end{enumerate}
	Thus we see that the subdifferential is additive and we obtain the assumption.
\end{proof}
\begin{rem}
\label{rem:subdiff_nonneg_ell1inf}
~\\
	Clearly $P\in\partial \mathcal{R}\!\left(U\right)$ can
	be characterized as follows:
	\begin{align*}
		p_{ij}
		\begin{cases}
			= 0 & \text{if } i\notin I \text{ and } u_{ij}\neq0,\\
			\leq 0 & \text{if } i\notin I \text{ and } u_{ij}=0,\\
			= \omega_i & \text{if } i\in I \text{ and } u_{ij}>0, \\
			\in[-\omega_i,\omega_i] & \text{if } i\in I \text{ and }
			u_{ij}=0,\\
			= -\omega_i & \text{if } i\in I \text{ and } u_{ij}<0.
		\end{cases}
	\end{align*}
\end{rem}

\subsubsection{Source Conditions}

Knowing the characterization of the subgradient, we can state a condition, which
allows us to determine whether a certain solution to $AUB^T = W$ is
$\ell^{1,\infty}$-minimizing. We will call this condition a \emph{source
condition} as used in the inverse problem and error estimate literature e.g. in
\cite{Engl1996,Burger2004,Schuster2012}.
However, we would like to point out that similar conditions have been
called \emph{dual certificate} in the compressed sensing literature (c.f.
\cite{Zhang2012, Candes2010, Candes2011, Hsu2011}).
\begin{defi}
~\\
	We say that a solution $\widehat{U}$ of $A\widehat{U}B^T=W$ meets a source
	condition with respect to a proper, convex regularization functional $J$ if there exists
	a $Q$ such that $P = A^TQB \in \partial J(\widehat{U})$.
\end{defi}

The source condition of some $\widehat{U}$ with respect to $J$ is nothing but
the optimality condition for $\widehat{U}$ being a $J$-minimizing solution to
$A\widehat{U}B^T=W$.
\begin{lem}[cf. \cite{Burger2004}]
~\\
	Let $\widehat{U}$ with $A\widehat{U}B^T=W$ meet a source condition with respect
	to $J$.
	Then $\widehat{U}$ is a $J$-minimizing solution.
\end{lem}

Considering this, the next question naturally emerges for our characterization
of the subdifferential, i.e what kind of solutions $\widehat{U}$ to $A\widehat{U}B^T=W$
are likely to meet a source condition for $\ell^{1,\infty}$-regularization.
Particularly, we are interested in investigating how likely
$\ell^{0,\infty}$-minimizing solutions are to meet a source condition.\\
Due to the similarity between the subgradient of the $\ell^1$-norm and the
subgradient of $\ell^{1,\infty}$-norm at rows with index $i \in I$, we can make
the following simple observation:
\begin{lem}
~\\
	Let $\widehat{U}$ be an $\ell^1$-minimizing solution to $A\widehat{U}B^T=W$ for
	which we have $\sum_{j=1}^N |\hat{u}_{ij}| = \sum_{j=1}^N
	|\hat{u}_{mj}|$ for all $i,m\in\{1,\ldots,M\}$, then $\widehat{U}$ also is an
	$\ell^{1,\infty}$-minimizing solution.
\end{lem}
\begin{proof}
~\\
	The $\ell^1$-subgradient divided by the number of rows is an
	$\ell^{1,\infty}$-subgradient.
\end{proof}
The above lemma particularly shows that exact recovery criteria for the
properties of the sensing matrix $\text{kron}(B,A)$ (like the Restricted
Isometry Property \cite{Candes2005}, the Null Space Property
\cite{Donoho2001, Cohen2009} or the Mutual Incoherence
Property \cite{Donoho2001}), are sufficient for the exact recovery of
sparse solutions with the same $\ell^1$-norm in each row.

Of course, we do expect to recover more $\ell^{0,\infty}$-minimizing solutions
than just the ones with the same row-$\ell^1$-norm.
Looking at the characterization of the subdifferential (cf. Remark
\ref{rem:subdiff_nonneg_ell1inf}), we can observe that there are two cases that
pose much more severe restrictions, i.e. the equality constraints, than the two
other cases (which only lead to inequality constraints).
Thus we generally expect solutions, which require only a few of the equality
constraints to be more likely to meet a source condition.
As we can see equality constraints need to be met for nonzero elements, such
that the $\ell^{1,\infty}$-regularization prefers sparse solutions.
Additionally, the constraints are less restrictive if the corresponding row has
maximal $\ell^1$-norm.
Thus we expect those solutions to be likely to meet a source condition that
reach a maximal $\ell^1$-norm in as many rows as possible while being sparse.
Naturally, these solutions will be row-sparse, which further justifies the idea
that the $\ell^{1,\infty}$-norm can be used as a convex approximation of the
$\ell^{0,\infty}$-problem.

\subsection{Equivalence of Formulations}
\label{equivalence_of_formulations}

In this subsection we will show that the minimizers of \eqref{eq:allg_max_form}
and \eqref{eq:easy_problem} coincide under certain circumstances.
In order to do so, we examine how the regularization parameter $\alpha>0$ is
connected to the minimizer of \eqref{eq:allg_max_form}.
For this purpose we first prove the continuity and monotonicity of
\begin{align}
	\alpha\mapsto\max\limits_{i\in\{1,\ldots,M\}}\sum\limits_{j=1}^N u_{ij}(\alpha)
\;.
	\label{eq:pos_ell1inf_depending_on_alpha}
\end{align}
Afterwards, we shall investigate the meaning of $\tilde{v}$ in
\eqref{eq:easy_problem} and its connection to the minimizer of $F(U)$ with
minimal $\ell^{1,\infty}$-norm. Analyzing the limits of
\eqref{eq:pos_ell1inf_depending_on_alpha} leads us to the main result of
this subsection and the connection between the two problems
\eqref{eq:allg_max_form} and \eqref{eq:easy_problem}.

\begin{rem}
~\\
	For most of the proofs in this subsection we require $F(U)$ to be continuous.
	Thus most of the results are not useful for the constrained problem
	\eqref{eq:pos_const_prob}, since
	\begin{align*}
		F(U) =
		\begin{cases}
			0 & \text{if } AUB^T = W, \\
			\infty & \text{else,}
		\end{cases}
	\end{align*}
	is not continuous. Nevertheless, in reality we have to deal with noisy data
	anyway and thus we only want to implement the unconstrained problem
	\eqref{eq:pos_unconst_prob}. This will become easier, since we can simply use
	its reformulation \eqref{eq:easy_problem}, which we will summarize in Theorem
	\ref{th:problem_implementation}.
\end{rem}
\noindent For this subsection we define the functional
$J_{\alpha}\!:\!G\longrightarrow\R{+}$ via
\begin{align}
	J_{\alpha}\left(U\right):= F(U) +
	\alpha\max\limits_{i\in\left\{1,\ldots,M\right\}}\sum\limits_{j=1}^N u_{ij} \; .
\label{eq:fctal_for_proof}
\end{align}

\begin{lem}[Continuity of \eqref{eq:pos_ell1inf_depending_on_alpha}]
\label{th:continuity_v_alpha}
~\\
	Let $F\!:\!\R{M\times N}\longrightarrow\R{+}$ be a
	convex continuous functional with bounded sublevel sets. 
	Let $U(\alpha)\in G$ be a minimizer of $J_{\alpha}$ with $\lVert
	U(\alpha)\rVert_{1,\infty}$ minimal. Then any other minimizer $V \in G$ of $J_\alpha$ satisfies
  \begin{equation}
			\max\limits_{i\in\left\{1,\ldots,M\right\}}\sum\limits_{j=1}^N v_{ij} =	\max\limits_{i\in\left\{1,\ldots,M\right\}}\sum\limits_{j=1}^N u_{ij}(\alpha)
	\end{equation}
	and the consequently well-defined map
	$\alpha\mapsto
	\max\limits_{i\in\left\{1,\ldots,M\right\}}\sum\limits_{j=1}^N u_{ij}(\alpha)$
	is a continuous function.
\end{lem}
\begin{proof}
~\\
	Analogous to the arguments in \cite{Burger2004} we conclude that the Bregman
	distance $D_R(U(\alpha),V)$ vanishes, with the convex regularization functional
	$R(V) = \max\limits_{i\in\left\{1,\ldots,M\right\}}\sum\limits_{j=1}^N v_{ij}$.
	The one-homogeneity of $R$ then immediately implies $R(V) = R(U(\alpha))$. Note
	further that due to the minimizing property we also have $J_\alpha(V) =
	J_\alpha(U(\alpha))$.

	Let $U_k$ be a minimizer of $J_{\alpha_k}$ and let
	$\alpha_k\rightarrow\alpha$ be a sequence of regularization parameters. Due to
	boundedness, we are able to find subsequences $U_{k_l}\rightarrow U$.
	Because of the convergence of the subsequences $R(U_k)$ and $J_{\alpha_K}(U_k)$
	and the uniqueness of their limits, the limits of all subsequences are equal
	and we obtain convergence of the whole sequences $R(U_k) \rightarrow R(U)$ and
	$J_{\alpha_k}(U_k) \rightarrow J(U)$.
	We proceed by contradiction and claim that $U$ is not a minimizer of
	$J_{\alpha}$.
	In this case there would exist a $\widetilde{U}$ with
	\begin{align}
	\label{eq:false_assumption}
		J_{\alpha}(\widetilde{U}) < J_{\alpha}\left(U\right) \; .
	\end{align}
	Let us consider
	\begin{align*}
		J_{\alpha}\left(\frac{\alpha_k}{\alpha}\widetilde{U}\right) 
		&= F\left(\frac{\alpha_k}{\alpha}\widetilde{U}\right) +
		\alpha_k\max\limits_{i\in\left\{1,\ldots,M\right\}}\sum\limits_{j=1}^N
		\tilde{u}_{ij} \\
		&= J_{\alpha_k}(\widetilde{U}) +
		F\left(\frac{\alpha_k}{\alpha}\widetilde{U}\right) - F(\widetilde{U}) \; .
	\end{align*}
	The continuity of $F$ yields
	\begin{align*}
		F\left(\frac{\alpha_k}{\alpha}\widetilde{U}\right) - F(\widetilde{U})
		\rightarrow 0
	\end{align*}
	for $k\rightarrow\infty$.
	Since $J$ is continuous, we obtain that
	\begin{align*}
		J_{\alpha_k}(\widetilde{U}) &\rightarrow J_{\alpha}(\widetilde{U})
		\qquad\text{and} \\
		J_{\alpha_k}(U_k) &\rightarrow J_{\alpha}(U)
	\end{align*}
	hold for $k\rightarrow\infty$. By using \eqref{eq:false_assumption},
	we see that the inequality $$ J_{\alpha_k}(\widetilde{U})
	< J_{\alpha_k}(U_k) $$ has to hold as well.
	This is a contradiction to the assumption that $U_k$ is a minimizer of
	$J_{\alpha_k}$.
	Hence $U$ has to be a minimizer of $J_{\alpha}$ and we see that
	$\alpha\mapsto \alpha\max\limits_{i\in\left\{1,\ldots,M\right\}}\sum\limits_{j=1}^N
	u_{ij}(\alpha)$ is continuous.
	Thus we also know that \eqref{eq:pos_ell1inf_depending_on_alpha}
	is continuous on $\left(0,\infty\right)$.
\end{proof}

Another well-known result is the monotonicity of the regularization (cf.
\cite{Tikhonov1977} for a more general statement):
\begin{lem}[Monotonicity of \eqref{eq:pos_ell1inf_depending_on_alpha}]
~\\
	Let $F\!:\!\R{M\times N}\longrightarrow\R{+}$ be a convex continuous functional
	with bounded sublevel sets.
	Let $U(\alpha)\in G$ be a minimizer of $J_{\alpha}$ with $\lVert
	U(\alpha)\rVert_{1,\infty}$ minimal.\\
	Then $\alpha\mapsto \max\limits_{i\in\left\{1,\ldots,M\right\}}\sum\limits_{j=1}^N
	u_{ij}(\alpha)$ is a monotonically decreasing function.
\end{lem}

\begin{lem}
\label{th:at_least_one_i}
~\\
	Let $F\!:\!\R{M\times N}\longrightarrow\R{+}$ be a
	convex continuous functional with bounded sublevel sets.
	Let $\bar{U}\in G$ be a solution of \eqref{eq:easy_problem} such that
	$\sum\limits_{j=1}^N \bar{u}_{ij} < \tilde{v}$ holds for all
	$i\in\{1,\ldots,M\}$.\\
	Then we have $\tilde{v}>\lVert\widehat{U}\rVert_{1,\infty}$, where
	$\widehat{U}\in G$ is a minimizer of $F(U)$ with $\lVert U\rVert_{1,\infty}$
	minimal and vice versa.
\end{lem} 
\begin{proof}
~\\
	Let be $\tilde{v}>\lVert\widehat{U}\rVert_{1,\infty}$, then $\widehat{U}$ is
	feasible for \eqref{eq:easy_problem} and obviously a minimizer as well. \\
	Let be $\tilde{v}\leq\lVert\widehat{U}\rVert_{1,\infty}$. Then in case that
	$\sum\limits_{j=1}^N\bar{u}_{ij}<\tilde{v}$ holds for all
	$i\in\{1,\ldots,M\}$, we obviously have
	$\bar{U}\not\equiv\widehat{U}$ and thus we obtain
	\begin{align*}
		F(\bar{U})>F(\widehat{U}) \; ,
	\end{align*}
	since $\widehat{U}$ is a minimizer of $F$ with $\lVert\widehat{U}\rVert_{1,\infty}$
	minimal.
	Due to convexity, we obtain
	\begin{align*}
		F(\varepsilon\widehat{U}+(1-\varepsilon)\bar{U}) \leq \varepsilon F(\widehat{U}) +
		(1-\varepsilon)F(\bar{U}) < F(\bar{U})
	\end{align*}
	and for small $\varepsilon$ we have
	\begin{align*}
		\varepsilon\sum\limits_{j=1}^N \hat{u}_{ij} +
		(1-\varepsilon)\sum\limits_{j=1}^N\bar{u}_{ij} \leq \tilde{v} \; .
	\end{align*}
	Thus we have found an element with smaller value of $F$ as the minimizer
	$\bar{U}$, which is a contradiction.
\end{proof}

\begin{lem}[Limits of \eqref{eq:pos_ell1inf_depending_on_alpha}]
\label{th:limits_of_ell1inf}
~\\
	Let $F\!:\!\R{M\times N}\longrightarrow\R{+}$ be a
	convex continuous functional with bounded sublevel sets.
	Then we have
	\begin{center}
		\begin{tabular}{l l l l}
 			$\max\limits_{i\in\left\{1,\ldots,M\right\}}\sum\limits_{j=1}^N
 			u_{ij}(\alpha) \rightarrow 0$ & for & $\alpha\rightarrow\infty$ & \quad and
 			\\
 			$\max\limits_{i\in\left\{1,\ldots,M\right\}}\sum\limits_{j=1}^N
 			u_{ij}(\alpha) \rightarrow \lVert\widehat{U}\rVert_{1,\infty}$ & for
 			& $\alpha\rightarrow 0$ &
		\end{tabular}
	\end{center}		
	with $\widehat{U}\in G$ being a minimizer of $F$ with $\lVert
	U\rVert_{1,\infty}$ minimal.
\end{lem}
\begin{proof}
~\\
	Let $U(\alpha)$ be a minimizer of $J_{\alpha}$ as proposed in
	\eqref{eq:fctal_for_proof}.
	\begin{enumerate}
		\item Consider the case of $\alpha\rightarrow\infty$.
			$U\equiv 0$ is feasible for $J_{\alpha}$, thus we obtain
			\begin{align*}
				F\left(U(\alpha)\right) +
				\alpha\max\limits_{i\in\left\{1,\ldots,M\right\}}\sum\limits_{j=1}^N
				u_{ij}(\alpha) \leq F(0) \; .
			\end{align*}
			Therefore,
			$\alpha\!\max\limits_{i\in\left\{1,\ldots,M\right\}}\sum\limits_{j=1}^N
			u_{ij}(\alpha)$ is bounded by $F(0)$ and we have
			\begin{align*}
				\max\limits_{i\in\left\{1,\ldots,M\right\}}\sum\limits_{j=1}^N
				u_{ij}(\alpha) \rightarrow 0 \quad\text{for}\quad
				\alpha\rightarrow\infty \; .
			\end{align*}
		\item Now consider the case of $\alpha\rightarrow 0$.
			We can find a subsequence $\alpha_k\rightarrow 0$ such that 
			\begin{align*}
				U\left(\alpha_k\right)\rightarrow\widehat{U}
			\end{align*}
			holds, where $\widehat{U}$ is a minimizer of $F$ with
			$\left\|U\right\|_{1,\infty}$ minimal.
			Obviously $\widehat{U}$ is feasible for $J_{\alpha_k}$.
			Hence we obtain
			\begin{align*}
				F\left(U(\alpha_k)\right) +
				\alpha_k\max\limits_{i\in\left\{1,\ldots,M\right\}}\sum\limits_{j=1}^N
				u_{ij}(\alpha_k) \leq F(\widehat{U}) +
				\alpha_k\lVert\widehat{U}\rVert_{1,\infty} \; .
			\end{align*}
			Since $\widehat{U}$ is a minimizer of $F$, it has to hold that
			$F(U(\alpha_k))\geq F(\widehat{U})$ and thus we obtain
			\begin{align*}
				\max\limits_{i\in\left\{1,\ldots,M\right\}}\sum\limits_{j=1}^N
				u_{ij}(\alpha_k) \leq \lVert\widehat{U}\rVert_{1,\infty} \; .
			\end{align*}
			Obviously $U(\alpha)$ is feasible for \eqref{eq:easy_problem} with
			$$\tilde{v}=\max\limits_{i\in\left\{1,\ldots,M\right\}}\sum\limits_{j=1}^N
			u_{ij}(\alpha_k) \; .$$
			Then Lemma \ref{th:at_least_one_i} yields
			\begin{align*}
				\sum\limits_{j=1}^N u_{ij}(\alpha_k) =
				\max\limits_{i\in\left\{1,\ldots,M\right\}}\sum\limits_{j=1}^N
				u_{ij}(\alpha_k) = \lVert U(\alpha_k)\rVert_{1,\infty} \quad\text{ for some
				} i\in \left\{1,\ldots,M\right\} \; .
			\end{align*}
			Due to the lower semi-continuity of the norm, we obtain
			\begin{align*}
				\liminf_{\alpha_k\rightarrow 0} \;
				\max\limits_{i\in\left\{1,\ldots,M\right\}}\sum\limits_{j=1}^N u_{ij}(\alpha_k) =
				\liminf_{\alpha_k\rightarrow 0} \; \lVert U(\alpha_k)\rVert_{1,\infty}
				\geq \lVert\widehat{U}\rVert_{1,\infty}
			\end{align*}
			and finally we have
			\begin{align*}
				\max\limits_{i\in\left\{1,\ldots,M\right\}}\sum\limits_{j=1}^N
				u_{ij}(\alpha) \rightarrow \lVert\widehat{U}\rVert_{1,\infty}
				\quad\text{for}\quad \alpha\rightarrow 0 \; .
			\end{align*}
	\end{enumerate}
\end{proof}
\noindent
In Figure \ref{fig:bound_for_l1inf} we see an example on how the function
$$v(\alpha):=\max\limits_{i\in\{1,\ldots,M\}}\sum\limits_{j=1}^N
u_{ij}(\alpha)$$ could look, where $U(\alpha)\in G$ is a minimizer of
$J_{\alpha}$ with $\lVert U(\alpha)\rVert_{1,\infty}$ minimal.\\
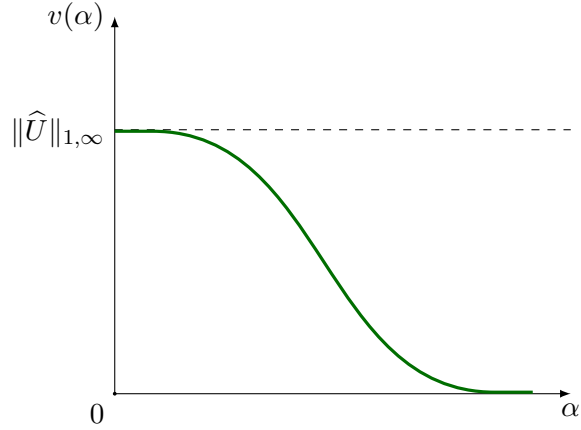
\begin{figure}[htb]
	\begin{center}
  		\begin{tikzpicture}[>=latex]
	\draw [->] (0,0) node [below left] {$0$} -- (0,5) node [left] {$v(\alpha)$};
	\draw [->] (0,0) -- (6,0) node [below] {$\alpha$};
	\filldraw (0,0) circle (0.5pt);
	\draw [dashed] (0,3.5) node [left] {$\|\widehat{U}\|_{1,\infty}$} -- (6,3.5);
	\draw [color=mygreen,line width=1.2pt] (0,3.48) -- (0.5,3.48) .. controls
	(2.75,3.48) and (2.75,0.02) ..
	(5,0.02) -- (5.5,0.02);
\end{tikzpicture}
		\caption[Relation between the regularization parameter and the nonnegative
		$\ell^{1,\infty}$-norm of the corresponding minimizer]{Illustration of the
		relation between the regularization parameter $\alpha$ and the nonnegative
		$\ell^{1,\infty}$-norm $v(\alpha)$ of the corresponding minimizer as defined
		above}
		\label{fig:bound_for_l1inf}
	\end{center}
\end{figure}
~\\
\begin{rem}
~\\
	In case that $\partial F(0)\neq \emptyset$ holds, we have
	$$\max\limits_{i\in\left\{1,\ldots,M\right\}}\sum\limits_{j=1}^N
	u_{ij}(\alpha) = 0$$ already for $\alpha<\infty$, but large enough.\\
	We see this by considering $P_0\in\partial F(0)$ and then selecting
	$P=-\frac{1}{\alpha}P_0$. Here we choose $\alpha$ large enough that we
	have $\norm{P}_{\infty,1}<1$ and obtain
	$$ \scalpr{P,V}_{F} \leq \norm{P}_{\infty,1}J(V) \leq J(V) \; .$$
	Thus $P$ is a subgradient at $U(\alpha)=0$ and we see that the optimality
	condition is fulfilled.
\end{rem}

\noindent Finally, we can conclude the following essential statement:

\begin{thm}[Connection of the Solutions of \eqref{eq:allg_U_v_form} and
\eqref{eq:easy_problem}]
\label{th:problem_implementation}
~\\
	Let $F\!:\!\R{M\times N}\longrightarrow\R{+}$ be a convex continuous functional
	with bounded sublevel sets.
	Let be $\tilde{v}\in (0,\lVert\widehat{U}\rVert_{1,\infty})$, where
	$\widehat{U}$ is a minimizer of $F(U)$ with $\lVert U\rVert_{1,\infty}$ minimal. 
	Let $\bar{U}\in G$ be a solution of
	\begin{align}
		\min\limits_{U\in G} \; F(U) \quad\text{s. t.}\quad \sum\limits_{j=1}^N
		u_{ij}\leq\tilde{v} \;.
	\tag{\ref{eq:easy_problem}}
	\end{align}
	Then there exists an $\alpha>0$ such that $\bar{U}$ is a solution of
	\begin{align}
		\min\limits_{U\in G} \; F(U) +
		\alpha\max\limits_{i\in\{1,\ldots,M\}}\sum\limits_{j=1}^N u_{ij} \; .
	\tag{\ref{eq:allg_max_form}}
	\end{align}
\end{thm}

\begin{rem}
~\\
	If \;$U$ is a solution of \eqref{eq:easy_problem}, we can directly decide,
	whether there exists an $\alpha$ for this problem, i.e. in the case of
	\begin{align*}
		\tilde{v} = \lVert U \rVert_{1,\infty} \; .
	\end{align*}
\end{rem}

\begin{thm}[Existence of a Solution of \eqref{eq:easy_problem}]
~\\
	Let $F\!:\!\R{M\times N}\longrightarrow\R{+}$ be a convex continuous functional
	with bounded sublevel sets.
	Then there exists a minimizer of \eqref{eq:easy_problem}.
\end{thm}
\begin{proof}
~\\
	We can write \eqref{eq:easy_problem} using the characteristic function, i.e.
	\begin{align*}
		\min_{U\in G}\; F(U) + \chi_{\left\{U\in\R{M\times N}\,|\,\sum_{j=1}^N
		u_{ij}\,\leq\,\tilde{v}\;\forall\,i\in\{1,\ldots,M\}\right\}} \; .
	\end{align*}
	For $\xi\in\R{}$ we consider the sublevel set
	\begin{align*}
		\mathcal{S}_{\xi} = \left\{U\in G\,|\, F(U) + \chi_{\left\{U\in\R{M\times N}\,|\,\sum_{j=1}^N
		u_{ij}\,\leq\,\tilde{v}\;\forall\,i\in\{1,\ldots,M\}\right\}} \leq \xi
		\right\} \; .
	\end{align*}
	Since $F$ is continuous and $\chi(0)=0$ holds, we have $0\in \mathcal{S}_{\xi}$ and thus
	$\mathcal{S}_{\xi}$ is not empty.
	Furthermore, $\mathcal{S}_{\xi}$ is bounded, since the norm of $U$ is bounded.
	Additionally, the functional stays lower semicontinuous and we obtain the
	existence of a minimizer of \eqref{eq:easy_problem}.
\end{proof}

\noindent By using Theorem \ref{th:problem_implementation} for problem
\eqref{eq:Uv_unconstr_prob}, we obtain
\begin{align}
	\min_{U\in G} \frac{1}{2}\left\|\gamma\left(AUB^T-W\right)\right\|_F^2 +
	\beta\sum\limits_{i=1}^M\sum\limits_{j=1}^N u_{ij} \quad \text{ s.t. }
	\quad \sum\limits_{j=1}^N u_{ij}\leq \tilde{v} \; .
	\label{eq:prob_for_algorithm}
\end{align}
Note that we need to look for a suitable regularization parameter in the
implementation anyway. Thus we can instead determine a suitable $\tilde{v}$ and
obtain an easier optimization problem.

\subsection{Asymptotic 1-Sparsity}
\label{asymptotic_one-sparsity}

In this section we consider the asymptotics of our reformulated problem
\eqref{eq:prob_for_algorithm} with $\beta = 0$ and $\gamma=1$, i.e.
\begin{align}
	\min_{U\in G} \;\frac{1}{2}\left\|AUB^T-W\right\|_F^2 \st \sum_{j=1}^N
	u_{ij}\leq\tilde{v} \; ,
	\label{eq:unconstr_prob_for_impl_without_beta}
\end{align}
where the $\ell^{1,\infty}$-regularization gains more and more influence, which
means that the regularization parameter $\tilde{v}$ shall go to zero.
In the limit case we observe that we indeed obtain a special kind of sparsity in
every row, i.e. we are able to determine the number of nonzero elements in
each row.
In a special case we are moreover able to locate the nonzero entries of a
solution of \eqref{eq:unconstr_prob_for_impl_without_beta} even if $\tilde{v}$
is nonzero but small enough.

In order to analyze \eqref{eq:unconstr_prob_for_impl_without_beta}
asymptotically, we consider the rescaling $X:=\tilde{v}^{-1}U \Leftrightarrow
U=\tilde{v}X$ and obtain the new variational problem
\begin{align}
	\min_{X\in G} \;\frac{1}{2}\norm{\tilde{v}AXB^T-W}_F^2 \st \sum_{j=1}^N
	x_{ij}\leq 1 \; .
	\label{eq:rescaled_problem}
\end{align}
Let us now analyze the structure of a solution of \eqref{eq:rescaled_problem}
for $\tilde{v}\rightarrow 0$.
\begin{thm}
\label{th:rescaled_asymptotic_sparsity}
~\\
	Let $k_i$ be the number of maxima in the $i$th row of $Y:=A^TWB$
	and let $X(\tilde{v})$ be a minimizer of \eqref{eq:rescaled_problem}.
	Then the $i$th row of
	\begin{align*}
		\bar{X} := \lim_{\tilde{v}\rightarrow 0} X(\tilde{v})
	\end{align*}
	is at most $k_i$-sparse.
\end{thm}
\begin{proof}
~\\
	After simplifying the norm and dividing by $\tilde{v}$ in
	\eqref{eq:rescaled_problem}, we can equivalently consider
	\begin{align*}
		\min_{X\in G} \;\frac{\tilde{v}}{2}\norm{AXB^T}_F^2 - \left\langle
		AXB^T,W\right\rangle_F \st \sum_{j=1}^N x_{ij}\leq 1 \; ,
	\end{align*}
	where $\left\langle
	C,D\right\rangle_F:=\sum\limits_{l=1}^L\sum\limits_{k=1}^Tc_{ij}d_{ij}$
	denotes the Frobenius inner product. For the case that we have
	$\tilde{v}\rightarrow 0$, the first summand tends to zero and thus
	\begin{align}
		\max_{X\in G} \;\left\langle X,A^TWB\right\rangle_F \st \sum_{j=1}^N
		x_{ij}\leq 1
		\label{eq:reformulated_limit_problem}
	\end{align}
	holds. With the above definition of $Y$ we shall now consider
	\begin{align*}
		\max_{X\in G} \;\sum_{i=1}^M\sum_{j=1}^N y_{ij}x_{ij} \st \sum_{j=1}^N
		x_{ij}\leq 1 \; .
	\end{align*}
	Let $J_i$ be the column index set at which the maximum of the $i$th row of $Y$
	is reached, i.e.
	\begin{align*}
		J_i = \left\{n\in\{1,\ldots,N\}\,|\,y_{in}\geq
		y_{ij}\;\forall\,j\in\{1,\ldots,N\}\right\}
	\end{align*}
	for every $i\in\{1,\ldots,M\}$.
	Since we have the constraint that the row sum of $X$ should not exceed $1$,
	we obtain the inequality
	\begin{align*}
		\sum_{j=1}^N y_{ij}x_{ij}\leq y_{in} \qquad\forall\;n\in J_i \;
	\end{align*}
	for every $i\in\{1,\ldots,M\}$.
	Hence we see that for a solution of \eqref{eq:reformulated_limit_problem} has
	to hold
	\begin{align*}
		\sum_{n\in J_i}x_{in} = 1 \qquad\text{and}\qquad x_{ij} = 0
		\quad\forall\;j\notin J_i \; .
	\end{align*}
	This means that the $i$th row of the solution of
	\eqref{eq:reformulated_limit_problem} has at most $k_i$ nonzero entries.
	Thus the $i$th row of $\bar{X}$ is at most $k_i$-sparse, maybe even sparser.
\end{proof}
\begin{rem}
~\\
	In case that the $i$th row of $\bar{X}$ is $k_i$-sparse, the asymptotic
	solution $\bar{X}$ has nonzero entries at the same positions as $Y=A^TWB$ has
	its maxima in each row.\\
	Note that the row-maxima are not necessarily unique. However, in the
	case that for every $i\in\{1,\ldots,M\}$ the index set $J_i$ contains only one
	element, the rows of $\bar{X}$ are 1-sparse.
\end{rem}
Theorem \ref{th:rescaled_asymptotic_sparsity} raises the question, whether
there exists a small regularization parameter $\tilde{v}$, for which
$X(\tilde{v})$ is already $k_i$-sparse. In this case we could apply this
knowledge to the original problem
\eqref{eq:unconstr_prob_for_impl_without_beta}, which is not possible in the
limit case, since then $\bar{U}:=\lim\limits_{\tilde{v}\rightarrow
0}U(\tilde{v})$ would be equal to zero.
\begin{thm}
~\\
	Let the $\ell^2$-norm of the columns of $A\in\R{L\times M}$ and
	$B\in\R{T\times N}$ be nonzero, i.e.
	\begin{align*}
		\norm{a_{\cdot i}}_2 > 0 \quad\forall\; i\in\{1,\ldots,M\}
		\qquad\text{and}\qquad \norm{b_{\cdot j}}_2 > 0 \quad\forall\;
		j\in\{1,\ldots,N\} \; .
	\end{align*}
	Then there exists a regularization parameter $\tilde{v}>0$ such that the
	solution of \eqref{eq:unconstr_prob_for_impl_without_beta} has nonzero entries
	at the same positions as $Y:=A^TWB$ has row-maxima.
\end{thm}
\begin{proof}
~\\
	Let us consider the rescaled problem \eqref{eq:rescaled_problem}.
	After simplifying the norm and dividing by $\tilde{v}$, we consider
	equivalently
	\begin{align*}
		\min_{X\in G} \;\frac{\tilde{v}}{2}\norm{AXB^T}_F^2 - \left\langle
		X,A^TWB\right\rangle_F \st \sum_{j=1}^N x_{ij}\leq 1
	\end{align*}
	and thus we have
		\begin{align*}
		\min_{X\in G} \;\frac{\tilde{v}}{2}\sum_{l=1}^L\sum_{k=1}^T
		\left(\sum_{i=1}^M\sum_{j=1}^N a_{li}x_{ij}b_{kj}\right)^{\!\!2}
		-\sum_{i=1}^M\sum_{j=1}^N x_{ij}y_{ij} \st \sum_{j=1}^N x_{ij}\leq 1 \; ,
	\end{align*}
	where we use again $Y:=A^TWB$.
	The Lagrange functional reads as follows:
	\begin{align*}
		\mathcal{L}(X;\lambda,\mu) =
		&\;\frac{\tilde{v}}{2}\sum_{l=1}^L\sum_{k=1}^T \left(\sum_{i=1}^M\sum_{j=1}^N
		a_{li}x_{ij}b_{kj}\right)^{\!\!2} -\sum_{i=1}^M\sum_{j=1}^N x_{ij}y_{ij} \\
		&+ \sum_{i=1}^M\lambda_i\left(\sum_{j=1}^N x_{ij}-1\right) -
		\sum_{i=1}^M\sum_{j=1}^N\mu_{ij}x_{ij}
	\end{align*}
	with
	\begin{align*}
		\lambda_i \geq 0 \quad\text{and}\quad &\lambda_i\left(\sum_{j=1}^N
		x_{ij}-1\right) = 0 \; , \\
		\mu_{ij} \geq 0 \quad\text{and}\quad &\mu_{ij}x_{ij} = 0 \; .
	\end{align*}
	Let us now consider the optimality condition
	\begin{align*}
		0 = \partial_{x_{ij}}\mathcal{L} &= \tilde{v}x_{ij}\norm{a_{\cdot i}}_2^2
		\norm{b_{\cdot j}}_2^2 + \tilde{v}h_{ij} - y_{ij} + \lambda_i -\mu_{ij} \\
		\Leftrightarrow\quad \tilde{v}x_{ij} \norm{a_{\cdot i}}_2^2 \norm{b_{\cdot
		j}}_2^2 &= y_{ij} - \lambda_i + \mu_{ij} - \tilde{v}h_{ij} \; ,
	\end{align*}
	where $h_{ij}$ denotes the sum of the mixed terms resulting from the data term,
	which are independent from $x_{ij}$, i.e.
	\begin{align*}
		h_{ij} := \sum_{m\neq i}\sum_{n\neq j}\scalpr{a_{\cdot i},a_{\cdot
		m}}x_{mn}\scalpr{b_{\cdot n},b_{\cdot j}} + \norm{b_{\cdot j}}_2^2\sum_{m\neq
		i}\scalpr{a_{\cdot i},a_{\cdot m}}x_{mj} + \norm{a_{\cdot i}}_2^2\sum_{n\neq
		j}x_{in}\scalpr{b_{\cdot n},b_{\cdot j}} \; .
	\end{align*}
	Let now $\tilde{v}>0$ hold and let $J_i$ be the index set for which the entries
	of the $i$th row of the solution of
	\eqref{eq:unconstr_prob_for_impl_without_beta} are nonzero.
	We show that $y_{ij} > y_{in}$ holds for all $j\in J_i$ and for all $n\notin
	J_i$.
	In order to do so, we consider
	\begin{align*}
		0 = \partial_{x_{ij}}\mathcal{L} - \partial_{x_{in}}\mathcal{L}
		\qquad\forall\;j\in J_i,\;\forall\; n\notin J_i\; .
	\end{align*}
	We have $x_{ij}>0$ and $\mu_{ij}=0$, since it is $j\in J_i$. Furthermore, it
	holds that $x_{in}=0$ and $\mu_{in}\geq 0$, since we have $n\notin J_i$. Thus
	we obtain
	\begin{align*}
		0 \leq \mu_{in} = y_{ij} - y_{in} + \tilde{v}\left(h_{in}-h_{ij}\right) -
		\tilde{v}x_{ij}\norm{a_{\cdot i}}_2^2\norm{b_{\cdot j}}_2^2 
	\end{align*}
	and further 
	\begin{align*}
		0 < \tilde{v}x_{ij}\norm{a_{\cdot i}}_2^2\norm{b_{\cdot j}}_2^2 \leq 
		y_{ij} - y_{in} + \tilde{v}\left(h_{in}-h_{ij}\right) \; ,
	\end{align*}
	due to the fact that $\tilde{v},\;x_{ij},\;\norm{a_{\cdot
	i}}_2^2$ and $\norm{b_{\cdot j}}_2^2$ are positive.
	Hence it is left to show that
	\begin{align}
		y_{ij} > y_{in} + \tilde{v}d_{ijn}
		\label{eq:inequ_asymp_1_sparse}
	\end{align}
	holds, where we define $d_{ijn}:=h_{ij}-h_{in}$.
	This statement is obvious for $d_{ijn}\geq 0$.
	Let now $d_{ijn}<0$ hold.\\
	For $n\notin J_i$ we assume that $y_{in} =
	\max\limits_{\nu\in\{1,\ldots,N\}}y_{i\nu}$ holds. In addition let
	be $y_{in}\geq y_{ij} + 2\tilde{v}|d_{ijn}|$ for all $j\in J_i$.
	This is always possible, since we can choose $\tilde{v}>0$ small enough.
	With \eqref{eq:inequ_asymp_1_sparse} we have
	\begin{align*}
		y_{ij} > y_{in} - \tilde{v}|d_{ijn}|
	\end{align*}
	and thus we obtain
	\begin{align*}
		y_{ij} + \tilde{v}|d_{ijn}| > y_{in} \geq y_{ij} + 2\tilde{v}|d_{ijn}|
		\qquad \forall\;j \in J_i \; ,
	\end{align*}
	which is a contradiction.
	Hence we finally obtain that
	\begin{align*}
		y_{ij} > y_{in} \qquad \forall\;j\in J_i,\;\forall\;n\notin J_i
	\end{align*}
	has to hold and we see that the solution $X$ has nonzero entries at the same
	positions like $A^TWB$ has row-maxima, even for $\tilde{v}>0$ but small enough.
	Then obviously the same holds for the solution of
	\eqref{eq:unconstr_prob_for_impl_without_beta}.
\end{proof}

\section{Exact Recovery of Locally 1-Sparse Solutions}
\label{sec:scaling_conditions_for_exact_recovery_of_the_constrained_problem}

In this section we discuss the question of exact recovery for our model.

There already exist several conditions, which provide information about exact
reconstruction using linearly independent subdictionaries, see for instance
\cite{Fuchs2004, Tropp2006}. Unlike the case, where the basis vectors
are linearly independent, we consider the operator to be coherent, i.e. the
\emph{mutual incoherence parameter} (cf. \cite[p. 3]{Juditski2010})
\begin{align}
	\mu(B) := \max\limits_{i\neq j}\frac{\left|\left\langle
	b_i,b_j\right\rangle\right|}{\left\|b_i\right\|_2^2}
\label{eq:coherence_parameter} 
\end{align}
for $b_i$, $b_j$ being distant basis vectors, is large. In other words,
the vectors are very similar. This is a reasonable assumption for many
applications, see for instance the ones mentioned in the introduction as well as
in Subsection \ref{sec:applications}.

In Appendix \ref{sec:exact_reconstruction_of_a_delta_peak_in_1D} we gain some
understanding of necessary scaling conditions recovering locally $1$-sparse
solutions considering only one spacial dimension plus one additional dimension
(such as e.g. time) using problem \eqref{eq:varmodel1}.
We learn that if the solution is $1$-sparse in one spacial dimension plus the
additional dimension, the matrix $B\in\R{T\times N}$ has to meet the scaling
condition
\begin{align}
	\lVert \gamma b_n\rVert_{\ell^2}=1\qquad\text{and}\qquad\left|\left\langle
	\gamma b_n,\gamma b_m\right\rangle\right|\leq 1 \quad\text{for}\quad n\neq
	m
\label{eq:1D_peak_L2_exact_recovery}
\end{align}
with $\gamma\neq 0$ in order to recover 1-sparse
solutions.

\subsection{Lagrange Functional and Optimality Conditions}

In this subsection we introduce the Lagrange functional and optimality
conditions of problem \eqref{eq:Uv_constr_prob}, which we will need in the
further analysis.

We equivalently rewrite problem \eqref{eq:Uv_constr_prob} by writing the data
constraint for every $l$ and $k$, i.e.
\begin{align}
	\min\limits_{U\in G,\;v\in \R{+}} \;
	\beta\sum\limits_{i=1}^M\sum\limits_{j=1}^N u_{ij} + v \quad \text{ s.t. } \quad \alpha\sum\limits_{j=1}^N u_{ij}\leq v, \;
	\sum\limits_{i=1}^M\sum\limits_{j=1}^N a_{li}u_{ij}b_{kj} = w_{lk}
	\label{eq:Uv_constr_prob_new_data_term}
\end{align}
with $l\in\{1,\ldots,L\}$ and $k\in\{1,\ldots,T\}$.
For this problem the Lagrange functional reads as follows:
\begin{align}
	\begin{split}
		\mathcal{L}(v,u_{ij};\lambda,\mu,\eta) = \;
		&\beta\sum\limits_{i=1}^M\sum\limits_{j=1}^N u_{ij} + v +
		\sum\limits_{i=1}^M\lambda_i\left(\alpha\sum\limits_{j=1}^N u_{ij}-v\right) -
		\sum\limits_{i=1}^M\sum\limits_{j=1}^N\mu_{ij}u_{ij} \\ 
	 	&+
	 	\sum\limits_{l=1}^L\sum\limits_{k=1}^T\eta_{lk}\left(w_{lk} -
	 	\sum\limits_{i=1}^{M}\sum\limits_{j=1}^{N}a_{li}u_{ij}b_{kj}\right) \; ,
	\end{split}
	\label{eq:Uv_constr_prob_new_data_term_lagrange}
\end{align}
where $\lambda$, $\mu$ and $\eta$ are Lagrange parameters.
Now we are able to state the optimality conditions
\begin{align}
	0 &= \partial_v\mathcal{L}\;\; = 1-\sum\limits_{i=1}^M\lambda_i \tag{OPT1}
	\label{eq:opt_cond_lagrange1_constr}\; ,\\ 
	0 &= \partial_{u_{ij}}\mathcal{L} = \beta + \alpha\lambda_i - \mu_{ij} -
	\sum\limits_{l=1}^L\sum\limits_{k=1}^T\eta_{lk}a_{li}b_{kj}\;
	,
	\tag{OPT2}\label{eq:opt_cond_lagrange2_constr}
\end{align}
with the complementary conditions (cf. \cite[p. 305-306, Theorem
2.1.4]{Hiriart-Urruty1993})
\begin {align}
	\lambda_i\geq 0\quad &\text{ and }\quad
	\lambda_i\left(v-\alpha\sum\limits_{j=1}^N u_{ij}\right)=0 \; , \label{eq:KKT1}
	\\
	\mu_{ij}\geq 0\quad  &\text{ and }\quad \mu_{ij}u_{ij}=0 \; . \label{eq:KKT2} 
\end{align}

\subsection{Scaling Conditions for Exact Recovery of the Constrained
Problem}
\label{scaling_conditions_for_exact_recovery_of_the_constrained_problem}

On the basis of this insight, we examine under which assumptions a 1-sparse
solution of the constrained $\ell^{0,\infty}$-problem can be reconstructed
exactly by using the constrained $\ell^{1,\infty}$-$\ell^{1,1}$-minimization
\eqref{eq:Uv_constr_prob_new_data_term}.

We will see that the scaling condition \eqref{eq:1D_peak_L2_exact_recovery} in a
slightly reformulated way is a sufficient condition for exact recovery.

\begin{thm}[{Recovery of Locally 1-Sparse Data}]
\label{th:recovery_of_1-sparse_data}
	~\\
	Let be $c_i\in\R{+}$ and let
	\begin{align*}
		\hat{u}_{ij} = \begin{cases} c_i & \text{if } \; j=J(i)\, , \\ 
									 0   & \text{if } \; j\neq J(i)\, , \end{cases}
	\end{align*}
	be the exact solution of the constrained non-negative $\ell^{0,\infty}$-problem
	\begin{align}
		\min\limits_{U\in G} 
		\left(\max\limits_{i\in\left\{1,\ldots,M\right\}}\sum\limits_{j=1}^N
		u_{ij}^0\right) \st AUB^T=W \; ,
	\label{eq:constr_nonneg_ell0inf_prob}
	\end{align}
	with the definition $0^0:=0$. Here
	$J\!:\!\left\{1,...,M\right\}\longrightarrow\left\{1,...,N\right\}$ with
	$i\longmapsto J(i)$ denotes the function that maps every index
	$i\in\{1,\ldots,M\}$ to the index $J(i)\in\{1,\ldots,N\}$ of the corresponding
	basis vector, where the coefficient $\hat{u}_{iJ(i)}$ is unequal to zero, i.e.
	the rows of $\widehat{U}$ shall be 1-sparse and shall have their nonzero entry
	at the index $J(i)$.
	Let $A^T$ be surjective and let the scaling condition
	\begin{align}
		\left\|b_{J(i)}\right\|_2 = 1 \quad\text{and}\quad
		\left|\left\langle b_{J(i)},b_j\right\rangle\right| \leq
		1 \quad \forall \; j\in\left\{1,\ldots,N\right\}
	\label{eq:scaling_condition}
	\end{align}
	hold for all $i\in\left\{1,\ldots,M\right\}$.
	Then $\left(\widehat{U},\;\alpha \max\limits_{p\in\{1,\ldots,M\}} c_p\right)$ is a solution of
	\eqref{eq:Uv_constr_prob_new_data_term}.
\end{thm}
\begin{proof}
	~\\
	In order to proof Theorem \ref{th:recovery_of_1-sparse_data}, we have to show
	that there exist Lagrange parameters $\lambda\in\R{M}$, $\mu\in\R{M\times N}$
	and $\eta\in\R{L\times T}$ such that $\widehat{U}$ fulfills the optimality
	conditions \eqref{eq:opt_cond_lagrange1_constr} and
	\eqref{eq:opt_cond_lagrange2_constr} with respect to the complimentary
	conditions \eqref{eq:KKT1} and \eqref{eq:KKT2}.\\
	We choose the Lagrange parameters for all $i\in\{1,\ldots,M\}$ as follows:
	\begin{align*}
		\lambda_i = 
		\begin{cases} 
			\frac{1}{m} & \text{if } \; c_i = \frac{v}{\alpha}\, ,\\
					  0 & \text{if } \; c_i < \frac{v}{\alpha}\, ,
		\end{cases} 
	\end{align*}
	with $\frac{v}{\alpha}=\max\limits_{p\in\{1,\ldots,M\}} c_p$ and $m$ being the number of indices,
	for which holds $c_i=\frac{v}{\alpha}$,
	\begin{align*}
		\mu_{ij} = 
		\begin{cases}
			0 & \text{if } \; j=J(i)\,  ,\\
			\left(\alpha\lambda_i+\beta\right)\left(1-\sum\limits_{k=1}^T
			b_{kJ(i)}b_{kj}\right) & \text{if } \; j\neq J(i)\, ,
		\end{cases}
	\end{align*}
	and $\eta$ as solution of
	\begin{align}
		\sum\limits_{l=1}^L a_{li}\eta_{lk} = \left(\alpha\lambda_i + \beta\right)
		b_{kJ(i)} \qquad\forall\; i\in\left\{1,\ldots,M\right\}, \;\forall\;
		k\in\left\{1,\ldots,T\right\} \; .
	\label{eq:lagrange_param_assumption}
	\end{align}
	Note that \eqref{eq:lagrange_param_assumption} is solvable, since $A^T$ is
	surjective.
	\begin{enumerate}
	  \item Let us show that \eqref{eq:opt_cond_lagrange1_constr} and
	  	\eqref{eq:KKT1} hold for $\widehat{U}$:
  		\begin{itemize}
  		  \item[a)] Obviously we have
  		  	\begin{align*}
  		  		\sum\limits_{i=1}^M\lambda_i = \underset{\left.
				c_i=\frac{v}{\alpha}\right\}}{\sum\limits_{i\in\left\{1,\ldots,M\right|}}\frac{1}{m}
				= m\frac{1}{m} = 1 \; .
  		  	\end{align*}
			Thus \eqref{eq:opt_cond_lagrange1_constr} is fulfilled.
  		  \item[b)] In case that $c_i<\frac{v}{\alpha}$ holds, we see that
  		  \eqref{eq:KKT1} is trivially fulfilled. Hence let be $c_i =
  		  \frac{v}{\alpha}$. We consider
	  		\begin{align*}
	  			\lambda_i\left(v-\alpha\sum_{j=1}^N\hat{u}_{ij}\right) =
	  			\frac{1}{m}(v-\alpha c_i) = \frac{1}{m}(v-\alpha\frac{v}{\alpha}) = 0
	  		\end{align*}
	  		and observe that \eqref{eq:KKT1} is fulfilled as well.
  		\end{itemize}
	  \item Let us now show that \eqref{eq:opt_cond_lagrange2_constr} and
	  	\eqref{eq:KKT2} hold for $\widehat{U}$:
	  	\begin{itemize}
	  	   \item[a)] In case that $j=J(i)$ holds, we obtain $\hat{u}_{ij}=c_i$ and
			 $\mu_{iJ(i)} = 0$. Thus \eqref{eq:KKT2} is obviously fulfilled.
			 The other case, i.e. $j\neq J(i)$, yields $\hat{u}_{ij}=0$ and $\mu_{ij} =
			 \left(\alpha\lambda_i+\beta\right)\left(1-\sum\limits_{k=1}^T b_{kJ(i)}b_{kj}\right)$. Since \eqref{eq:scaling_condition} has to hold, we
			obtain $\mu_{ij}\geq 0$ and we observe that in this case \eqref{eq:KKT2}
			is fulfilled as well.
		  \item[b)] Let again be $j=J(i)$. Then we obtain
		  	\begin{align*}
		  		\alpha\lambda_i + \beta -
		  		\sum\limits_{l=1}^L\sum\limits_{k=1}^T\eta_{lk}a_{li}b_{kJ(i)}
		  		= \left(1- \sum\limits_{k=1}^T b^2_{kJ(i)}\right)\left(\alpha\lambda_i +
		  		\beta\right) = 0
		  	\end{align*}
		  	by using the definitions of $\eta$ and $\mu$ and the scaling condition
		  	\eqref{eq:scaling_condition}. In this case
		  	\eqref{eq:opt_cond_lagrange2_constr} is fulfilled.\\
		  	Let us now consider $j\neq J(i)$. Then we have
		  	\begin{align*}
		  		&\alpha\lambda_i + \beta -
		  		\sum\limits_{l=1}^L\sum\limits_{k=1}^T\eta_{lk}a_{li}b_{kJ(i)} - \mu_{ij}
		  		\\ =\; &\alpha\lambda_i + \beta -
		  		\sum\limits_{l=1}^L\sum\limits_{k=1}^T\eta_{lk}a_{li}b_{kJ(i)} - 
		  		\left(\alpha\lambda_i+\beta\right)\left(1-\sum\limits_{k=1}^T
				b_{kJ(i)}b_{kj}\right) \\
				=\; &0 \; ,
		  	\end{align*}
		  	where we use the definition of $\mu$. Thus we see that in this case
		  	\eqref{eq:opt_cond_lagrange2_constr} is fulfilled as well.
	  	\end{itemize}
	\end{enumerate}
	In summary, we see that there exist Lagrange parameters such that $\widehat{U}$
	fulfills the optimality conditions and complementary conditions of
	\eqref{eq:Uv_constr_prob_new_data_term}. Thus we obtain the assertion.
\end{proof}
All in all, we found a condition for exact recovery of solutions of the
constrained $\ell^{0,\infty}$-problem, which contain 1-sparse rows, using the
constrained problem \eqref{eq:Uv_constr_prob_new_data_term} for the
reconstruction, i.e. \eqref{eq:scaling_condition} has to hold.
\begin{rem}
~\\
	We need to assume that $A^T$ is surjective in order to solve
	\eqref{eq:lagrange_param_assumption}. Unfortunately, if $A^T$ is surjective,
	then $A$ is injective and thus we could easier consider $UB^T = A^\dagger W$,
	where $A^\dagger$ is the pseudoinverse of $A$.
\end{rem}
Let us now consider an example of the extremest under-determined case, i.e.
where we have $L=1$.
\begin{thm}
\label{th:ex_rec_not_possible}
~\\
	Let be $\beta=0$ and $A\in\R{1\times M}$ with $M>1$ and $a_i\neq 0$ for
	every $i\in\{1,\ldots,M\}$. Let
	\begin{align*}
		\hat{u}_{ij} = \begin{cases} c_i & \text{if } \; j=J(i)\, , \\ 
									 0   & \text{if } \; j\neq J(i)\, , \end{cases}
	\end{align*}
	be the exact solution of the nonnegative $\ell^{0,\infty}$-problem
	\eqref{eq:constr_nonneg_ell0inf_prob} with
	$J\!:\!\left\{1,...,M\right\}\longrightarrow\left\{1,...,N\right\}$,
	$i\longmapsto J(i)$ mapping again every index $i$ to the index of the
	corresponding basis vector, where the coefficient is unequal to zero.
	Furthermore, let $m\in\{1,\ldots,M\}$ be a row-index, where $\widehat{U}$
	reaches its maximum, i.e. we have $c_m = \max\limits_{p\in\{1,\ldots,M\}} c_p =
	\frac{v}{\alpha}$.\\
	In case the exact solution $\widehat{U}$ contains a row-vector $u_i$, which
	has its nonzero entry at the same position as $u_m$, i.e. $J(i) = J(m)$, but
	their entries differ, i.e. $c_i < c_m$, then exact recovery of $\widehat{U}$
	using the nonnegative $\ell^{1,\infty}$-problem
	\eqref{eq:Uv_constr_prob_new_data_term} for the reconstruction is not possible.
\end{thm}
\begin{proof}
~\\
	Let us suppose exact recovery were possible. Then there exist a
	$\lambda$, which fulfills \eqref{eq:opt_cond_lagrange1_constr} and \eqref{eq:KKT1}, a
	$\mu$, which fulfills \eqref{eq:KKT2} and an $\eta$ such that
	\eqref{eq:opt_cond_lagrange2_constr} is fulfilled.\\
	By considering the complementary condition \eqref{eq:KKT1} for $i
	\in\left\{1,\ldots,M \;\Big|\; c_i<\max\limits_{p\in\{1,\ldots,M\}}
	c_p\right\}$, we have
	\begin{align*}
		0 = \lambda_i\left(v-\alpha\sum\limits_{j=1}^N u_{ij}\right) 
		= \lambda_i\left(v-\alpha u_{iJ(i)}\right) = \lambda_i\underbrace{(v-\alpha
		c_i)}_{\neq 0} \; ,
	\end{align*}
	since it is $c_i<\frac{v}{\alpha}$. Thus $\lambda_i=0$ holds for every
	$i\in\left\{1,\ldots,M \;\Big|\; c_i<\max\limits_{p\in\{1,\ldots,M\}}
	c_p\right\}$.
	On the other hand with \eqref{eq:opt_cond_lagrange1_constr} we have
	\begin{align*}
		1 = \sum_{i=1}^M \lambda_i =
		\underset{\left. c_m=\max_p
		c_p\right\}}{\sum_{m\in\left\{1,\ldots,M|\right.}} \lambda_m \; ,
	\end{align*}
	which yields $\lambda_m>0$ for every $m\in\left\{1,\ldots,M \;\Big|\; c_m =
	\max\limits_{p\in\{1,\ldots,M\}} c_p\right\}$.\\
	Now let us consider \eqref{eq:opt_cond_lagrange2_constr} for $j=J(i)$ and
	$j=J(m)$, which then reads as follows:
	\begin{align*}
		\sum\limits_{k=1}^T\eta_{k}b_{kJ(m)} &= \alpha\frac{\lambda_i}{a_{i}}
		\qquad\text{and} \\
		\sum\limits_{k=1}^T\eta_{k}b_{kJ(m)} &= \alpha\frac{\lambda_m}{a_{m}} \; ,
	\end{align*}
	since we have $J(i) = J(m)$.
	Therefore, we obtain
	\begin{align*}
		a_m\lambda_i = a_i\lambda_m \; .
	\end{align*}
	This is a contradiction, since we have $\lambda_i = 0$, $\lambda_m>0$ and
	$a_i$ and $a_m$ are unequal to zero.\\
	Thus we observe that the 1-sparse $\ell^{0,\infty}$-solution $\hat{u}_{ij}$
	cannot be the solution of the $\ell^{1,\infty}$-problem
	\eqref{eq:Uv_constr_prob_new_data_term} and in this case exact recovery is not
	possible.
\end{proof}
\begin{rem}
~\\
	In the case of Theorem \ref{th:ex_rec_not_possible} there always exists a
	solution of \eqref{eq:Uv_constr_prob_new_data_term} and
	\eqref{eq:constr_nonneg_ell0inf_prob}, which has a nonzero element in just one
	row, i.e. $\widehat{U}$ itself is 1-sparse.
\end{rem}
Note that Theorem \ref{th:ex_rec_not_possible} does not state that the
reconstructed support is wrong. Hence we could still obtain important information from the
nonnegative $\ell^{1,\infty}$-reconstruction.\\
Furthermore, Theorem \ref{th:ex_rec_not_possible} does not apply for the case
where we have $\beta>0$, since in this case we obtain
$$ (\alpha\lambda_i + \beta)a_m = (\alpha\lambda_m + \beta)a_i $$ 
and thus we do not obtain a contradiction in the last step of the proof.
Thus Theorem \ref{th:ex_rec_not_possible} suggests the usage of $\beta>0$.

\section{Algorithms}
\label{sec:algorithms}

The numerical minimization of $\ell^{p,q}$-related regularization problems is
usually done by using a modified FOCUSS algorithm (cf. \cite{Rao1999,
Cotter2005}) if the problem includes an exact reconstruction, or in the noisy
case a thresholded Landweber iteration (both realizations may be found
in \cite{Kowalski2009}).
However, those algorithms are designed for $0<p,q\leq 2$ and since in our case
we have $q=\infty$, they are not suitable for minimizing
$\ell^{1,\infty}$-related problems.

In order to solve problem \eqref{eq:varmodel1} numerically, we use a different
approach and develop an algorithm for the solution of its reformulated
problem, i. e.
\begin{align}
	\min\limits_U \; \frac{1}{2}\left\|AUB^T - W\right\|_F^2 +
	\beta\sum\limits_{i=1}^M\sum\limits_{j=1}^N u_{ij} \;\text{ s. t. }\;
	\sum\limits_{j=1}^N u_{ij}\leq \tilde{v},\; \forall \; i, \; u_{ij}\geq 0 \quad
	\forall\; i,j \; .
	\tag{\ref{eq:prob_for_algorithm}}
\end{align}
For the sake of simplicity we exclude the weight $\gamma$ for now.

For the numerical solution of this problem, we use the alternate direction
method of multipliers (ADMM), which traces back to the works of Glowinski and
Tallec \cite{Glowinski1975} and Gabay and Mercier \cite{Gabay1976}.
It was furthermore subject of many other books and papers, including
\cite{Fortin1983a}, especially its chapters \cite{Fortin1983b} and
\cite{Gabay1983}, as well as \cite{Glowinski1987}, \cite{Tseng1991},
\cite{Fukushima1992}, \cite{Eckstein1993} and \cite{Chen1994}.
For the computation of reasonably simple sub-steps, we
split the problem twice.
In so doing we obtain
\begin{align*}
	\min\limits_{U,\; Z,\; D} \;\frac{1}{2}\left\|AZ - W\right\|_F^2 +
	\beta\sum\limits_{i=1}^M\sum\limits_{j=1}^N d_{ij} \; \text{ s. t. } \;
	&\sum\limits_{j=1}^N d_{ij}\leq \tilde{v},\; \forall i,\; d_{ij}\geq 0 \;
	\forall i,j \\
	& Z = UB^T, \; D = U\; .
\end{align*}
By using the Lagrange functional
\begin{align*}
 	\mathcal{L}\left(U,D,Z;\widetilde{P},\widetilde{Q} \right) =
 	\;&\frac{1}{2}\left\|AZ-W\right\|_F^2 +
 	\beta\sum\limits_{i=1}^M\sum\limits_{j=1}^N d_{ij} +
 	\left\langle\widetilde{P},U-D\right\rangle_F +
 	\left\langle\widetilde{Q},UB^T-Z\right\rangle_F \\
 	&\text{s.t.}\quad
 	\sum\limits_{j=1}^N d_{ij}\leq \tilde{v},\; \forall i,\; d_{ij}\geq 0
 	\;\forall i,j \; ,
\end{align*}
where $\widetilde{P}$ and $\widetilde{Q}$ are the dual variables, we
obtain the \emph{unscaled} augmented Lagrangian
\begin{align}
	\begin{split}
		\mathcal{L}_{un}^{\lambda,\mu}\left(U,D,Z;\widetilde{P},\widetilde{Q} \right) =
		\;&\frac{1}{2}\left\|AZ-W\right\|_F^2 +
	 	\beta\sum\limits_{i=1}^M\sum\limits_{j=1}^N d_{ij} +
	 	\left\langle\widetilde{P},U-D\right\rangle_F \\
	 	&+ \frac{\lambda}{2}\left\|U-D\right\|_F^2 
	 	+ \left\langle\widetilde{Q},UB^T-Z\right\rangle_F +
	 	\frac{\mu}{2}\left\|UB^T-Z\right\|_F^2 \\
	 	\quad &\text{s.t.}\quad \sum\limits_{j=1}^N d_{ij}\leq \tilde{v},\;
	 	\forall i,\; d_{ij}\geq 0 \; \forall i,j \; ,
	\end{split}
	\label{eq:unsc_aug_lagrange}
\end{align}
with Lagrange parameters $\lambda$, $\mu$. Since its handling is much easier we also
want to state the \emph{scaled} augmented Lagrangian, i.e.
\begin{align*}
	\mathcal{L}_{sc}^{\lambda,\mu}(U,D,Z;P,Q) = \;
	&\frac{1}{2}\left\|AZ - W\right\|_F^2 + \beta\sum\limits_{i=1}^M\sum\limits_{j=1}^N d_{ij} + \frac{\lambda}{2}\left\| U - D
	+ P\right\|_F^2 \\
	&+ \frac{\mu}{2}\left\| UB^T - Z + Q \right\|_F^2 \;\text{ s.
	t. }\; \sum\limits_{j=1}^N d_{ij}\leq \tilde{v},\;\forall i,\; d_{ij}\geq 0\;
	\forall i,j \; ,
\end{align*}
with the new scaled dual variables
$P:=\frac{\widetilde{P}}{\lambda}$ and $Q:=\frac{\widetilde{Q}}{\mu}$.
By using ADMM the algorithm reads as follows:
\begin{align*}
	U^{k+1} &= \argmin{U}{\mathcal{L}_{sc}^{\lambda,\mu}(U,Z^k,D^k;P^k,Q^k)}\\
	D^{k+1} &= \argmin{D}{\mathcal{L}_{sc}^{\lambda,\mu}(U^{k+1},Z^k,D;P^k,Q^k)}\\
	Z^{k+1} &=
	\argmin{Z}{\mathcal{L}_{sc}^{\lambda,\mu}(U^{k+1},Z,D^{k+1};P^k,Q^k)}\\
	P^{k+1} &= P^k - \left(D^{k+1} - U^{k+1}\right)\\
	Q^{k+1} &= Q^k - \left(Z^{k+1} - U^{k+1}B^T\right)
\end{align*}
For faster convergence we use a standard extension of ADMM in Subsection
\ref{adaptive_parameter_choice}, i.e. an adaptive parameter choice as proposed
in \cite[Subsection 3.4.1]{Boyd2010} with its derivation in \cite[Section
3.3]{Boyd2010}, which we will adapt to our problem.
Another advantage of this extension is that the performance becomes less
dependent on the initial choice of the penalty parameter.
In order to do so, we first propose the optimality conditions.

\subsection{Optimality Conditions}

We obtain the following primal feasibility conditions
 \begin{align}
	0 &= \partial_{P}\mathcal{L} = \lambda\left(U-D\right)
	\label{eq:primal_feas_cond1} \\
	0 &= \partial_{Q}\mathcal{L} = \mu\left(UB^T-Z\right)
	\label{eq:primal_feas_cond2}
\end{align}
and the dual feasibility conditions
\begin{align}
	0 &= \partial_{U}\mathcal{L} = \lambda P + \mu QB \label{eq:dual_feas_cond1} \\ 
	0 &\in \partial_{D}\mathcal{L} = \beta\mathds{1}_{m\times n} - \lambda P +
	\partial J(D) \label{eq:dual_feas_cond2} \\
	0 &= \partial_{Z}\mathcal{L} = A^T\left(AZ-W \right) - \mu Q
	\label{eq:dual_feas_cond3}
\end{align}
with
\begin{align}
	J(D) := \begin{cases}
		0 & \text{ if } \sum\limits_{j=1}^N d_{ij}\leq\tilde{v} \;\forall
		i,\;d_{ij}\geq 0\;\forall i,j \; ,\\
		\infty & \text{ else.}
	\end{cases}
	\label{eq:positive_ell1infty}
\end{align}
Since $U^{k+1}$ minimizes $\mathcal{L}^{\lambda,\mu}_{sc}(U,Z^k,D^k;P^k,Q^k)$ by
definition, we obtain
\begin{align*}
	0 \in \partial_U\mathcal{L}^{\lambda,\mu}_{sc} &=
	\lambda\left(U^{k+1}-D^k+P^k\right) + \mu\left(U^{k+1}B^T-Z^k+Q^k\right)B \\
	&= \lambda P^{k+1} + \lambda\left(D^{k+1}-D^k\right) + \mu Q^{k+1}B +
	\mu\left(Z^{k+1}-Z^k\right)B \; ,
\end{align*}
by using the definitions of $P^{k+1}$ and $Q^{k+1}$. This is equivalent to
\begin{align*}
	\lambda\left(D^k-D^{k+1}\right) + \mu\left(Z^k-Z^{k+1}\right)B \in \lambda
	P^{k+1} + \mu Q^{k+1}B \; ,
\end{align*}
where the right hand side is the first dual feasibility condition
\eqref{eq:dual_feas_cond1}. Therefore 
\begin{align}
	S^{k+1} = \lambda\left(D^k-D^{k+1}\right) + \mu\left(Z^k-Z^{k+1}\right)B
	\label{eq:dual_res}
\end{align}
can be seen as a dual residual for \eqref{eq:dual_feas_cond1}.
Analogically we consider $0 \in \partial_D\mathcal{L}^{\lambda,\mu}_{sc}$,
which yields that $P^{k+1}$ and $D^{k+1}$ always satisfy
\eqref{eq:dual_feas_cond2}. The same applies for $0 \in
\partial_Z\mathcal{L}^{\lambda,\mu}_{sc}$, where we see that $Q^{k+1}$ and
$Z^{k+1}$  
always satisfy \eqref{eq:dual_feas_cond3}. In addition we will
refer to
\begin{align}
	R_1^{k+1} &= \lambda\left(D^{k+1} - U^{k+1}\right) \qquad \text{ and }
	\label{eq:primal_res1} \\
	R_2^{k+1} &= \mu\left(Z^{k+1} - U^{k+1}B^T\right) 	\label{eq:primal_res2}
\end{align}
as the primal residuals at iteration $k+1$.

Obviously we obtain five optimality conditions (\ref{eq:primal_feas_cond1} -
\ref{eq:dual_feas_cond3}). We have seen that \eqref{eq:dual_feas_cond2} and
\eqref{eq:dual_feas_cond3} are always satisfied. The other three (\ref{eq:primal_feas_cond1} -
\ref{eq:dual_feas_cond1}) lead to the primal residuals \eqref{eq:primal_res1}
and \eqref{eq:primal_res2} and to the dual residual \eqref{eq:dual_res}, which
converge to zero as ADMM proceeds (cp. \cite[Appendix A, p. 106 et
seqq.]{Boyd2010}).

\subsection{Stopping Criteria}
\label{sec:stopping_criteria}

In analogy to \cite[Section 3.3.1]{Boyd2010} we derive
the stopping criteria for the algorithm. As shown in Appendix \ref{sec:inequality} the primal and dual
residuals can be related to a bound on the objective suboptimality of the current point
$Y^*$. Hence we obtain
\begin{align}
	\begin{split}
		&\frac{1}{2}\left\| AZ-W\right\|_F^2 +
		\beta\sum\limits_{i=1}^M\sum\limits_{j=1}^N d_{ij} + J(D) - Y^* \\
		\leq\; &\left\langle P^k,R_1^k\right\rangle_F + \left\langle
		Q^k,R_2^k\right\rangle_F + \left\langle U^k - U^*,S^k\right\rangle_F \; .
	\end{split}
	\label{eq:inequality}
\end{align}
We see that the residuals should be small in order to obtain small objective
suboptimality. Since we want to obtain a stopping criterion but $U^*$ is
unknown, we estimate that $\lVert U^k-U^*\rVert_F\leq d$ shall hold.
Thus we obtain
\begin{align*}
	&\;\frac{1}{2}\left\| AZ-W\right\|_F^2 +
	\beta\sum\limits_{i=1}^M\sum\limits_{j=1}^N d_{ij} + J(D) - Y^* \\
	\leq &\;\lVert P^k\rVert_F\lVert R_1^k\rVert_F +
	\lVert Q^k\rVert_F\lVert R_2^k\rVert_F + d\lVert S^k\rVert_F \; .
\end{align*}
It stands to reason that the primal and dual residual must be small, i.e.
\begin{align*}
	\lVert R_1^k\rVert_F \leq \varepsilon_1^{\text{pri}} \; , \qquad
	\lVert R_2^k\rVert_F \leq \varepsilon_2^{\text{pri}} \; , \qquad
	\lVert S^k\rVert_F \leq \varepsilon^{\text{dual}}  \; ,
\end{align*}
with tolerances $\varepsilon_{1,2}^{\text{pri}}>0$ and
$\varepsilon^{\text{dual}}>0$ for the feasibility conditions
(\ref{eq:primal_feas_cond1} - \ref{eq:dual_feas_cond1}), respectively.
Boyd et al suggest in \cite{Boyd2010} that those can be chosen via an
absolute and relative criterion, i.e.
\begin{align*}
 	\varepsilon_1^{\text{pri}} &=
 	\sqrt{MN}\;\varepsilon^{\text{abs}} +
 	\varepsilon^{\text{rel}}\max \left\{\lVert U^k\rVert_F,\lVert
 	D^k\rVert_F,0\right\} \; ,\\
	\varepsilon_2^{\text{pri}} &=
	\sqrt{MT}\;\varepsilon^{\text{abs}} +
	\varepsilon^{\text{rel}}\max \left\{\lVert U^kB^T\rVert_F,\lVert
	Z^k\rVert_F,0\right\} \; , \\
	\varepsilon^{\text{dual}} &= \sqrt{MN}\;\varepsilon^{\text{abs}}
	+ \varepsilon^{\text{rel}}\lVert\lambda P^k + \mu Q^k B\rVert_F \; ,
\end{align*}
where $\varepsilon^{\text{rel}} = 10^{-3}$ or $10^{-4}$ is a relative tolerance
and the absolute tolerance $\varepsilon^{\text{abs}}$ depends on the scale of
the typical variable values. Note that the factors $\sqrt{MN}$
and $\sqrt{MT}$ result from the fact that the Frobenius norms are
in $\R{M\times N}$ and $\R{M\times T}$, respectively.

\subsection{Adaptive Parameter Choice}
\label{adaptive_parameter_choice}

In order to extend the standard ADMM and to improve its convergence rate, we
vary the penalty parameters $\lambda^k$ and $\mu^k$ in each iteration as
proposed in \cite[Section 3.4.1]{Boyd2010}. 
This extension has been analyzed in \cite{Rockafellar1976} in the context of the
method of multipliers. There it has been shown that if the penalty parameters go
to infinity, superlinear convergence may be reached.
If we consider $\lambda$ and $\mu$ to become fixed after a finite number of
iterations, the fixed penalty parameter theory still applies, i.e. we obtain
convergence of the ADMM.

The following scheme is amongst others proposed in \cite{He2000},\cite{Wang2001}
and often works well.
\begin{align*}
	\lambda^{k+1} = 
	\begin{cases}
		\tau_1^{\text{incr}}\lambda^k & \text{if } \; \lVert R_1^k\rVert_F >
		\eta_1\lVert S^k\rVert_F , \\
		\frac{\lambda^k}{\tau_1^{\text{decr}}} & \text{if } \; \lVert S^k\rVert_F >
		\eta_1\lVert R_1^k\rVert_F , \\
		\lambda^k & \text{otherwise,}
	\end{cases} 
	\quad\text{and}\quad
	P^{k+1} =
	\begin{cases}
		\frac{P^k}{\tau_1^{\text{incr}}} & \text{if } \; \lVert R_1^k\rVert_F
		> \eta_1\lVert S^k\rVert_F , \\
		P^k\tau_1^{\text{decr}} & \text{if } \; \lVert S^k\rVert_F >
		\eta_1\lVert R_1^k\rVert_F , \\
		P^k & \text{otherwise,}
	\end{cases}
\end{align*}
\begin{align*}
	\mu^{k+1} &:= 
	\begin{cases}
		\tau_2^{\text{incr}}\mu^k & \text{if } \; \lVert R_2^k\rVert_F > \eta_2\lVert
		S^k\rVert_F , \\
		\frac{\mu^k}{\tau_2^{\text{decr}}} & \text{if } \; \lVert S^k\rVert_F >
		\eta_2\lVert R_1^k\rVert_F , \\
		\mu^k & \text{otherwise,}
	\end{cases}
	\quad\text{and}\quad
	Q^{k+1} =
	\begin{cases}
		\frac{Q^k}{\tau_2^{\text{incr}}} & \text{if } \; \lVert R_2^k\rVert_F
		> \eta_2\lVert S^k\rVert_F , \\
		Q^k\tau_2^{\text{decr}} & \text{if } \; \lVert S^k\rVert_F >
		\eta_2\lVert R_2^k\rVert_F , \\
		Q^k & \text{otherwise,}
	\end{cases}
\end{align*}
where $\eta_{1,2}>1$, $\tau_{1,2}^{\text{incr}}>1$,
$\tau_{1,2}^{\text{decr}}>1$.
Typical choices are $\eta_{1,2} = 10$ and $\tau_{1,2}^{\text{incr}} =
\tau_{1,2}^{\text{decr}} = 2$.
Note that the dual variables $P^k$ and $Q^k$ only have to be updated in the
scaled form.

\subsection{Solving the
\texorpdfstring{$\ell^{1,\infty}-\ell^{1,1}$}{l1infinity-l11}-Regularized
Problem}

\begin{algorithm}
\caption{\texorpdfstring{$\ell^{1,\infty}$-$\ell^1$}{l1inf-l1}-regularized
Problem via ADMM with Double Splitting}
\label{alg:L2_ell1_infty_via_ALM}
	\begin{algorithmic}[1]
			\State \textbf{Parameters:}
			$v>0,\;\beta>0,\;A\in\mathbb{R}^{L\times M},\;B\in\mathbb{R}^{T\times
			N},W\in\mathbb{R}^{L\times
			T},\;\eta_{1,2}>1,\;\tau_{1,2}^{\text{incr}}>1,\;\tau_{1,2}^{\text{decr}}>1,\;\varepsilon^{\text{rel}}
			= 10^{-3} \text{ or }
			10^{-4},\;\varepsilon^{\text{abs}}>0$ 
			\State \textbf{Initialization:}
			$U,Z,D,P,Q,S,R_1,R_2\equiv 0,\;\varepsilon_1^{\text{pri}} =
			\sqrt{MN}\;\varepsilon^{\text{abs}},\varepsilon_2^{\text{pri}} =
			\sqrt{MT}\;\varepsilon^{\text{abs}},$
			\State $\varepsilon^{\text{dual}} =	\sqrt{MN}\;\varepsilon^{\text{abs}}$ 
			\While{$\lVert
			R_1\rVert_F>\varepsilon_1^{\text{pri}}\;\text{\textbf{and}}\;\lVert
			R_2\rVert_F>\varepsilon_2^{\text{pri}}\;\text{\textbf{and}}\;\lVert S\rVert_F>\varepsilon^{\text{dual}}$}
 				\State $D^{\text{old}}=D$;
 				\State $Z^{\text{old}}=Z$;
 				\Statex								\Comment{Main Part}
 				\State $U = \left(\lambda\left(D-P\right)
 				+\mu\left(Z-Q\right)B\right)\left(\lambda I+\mu B^T B\right)^{-1}$;
 				\State $D =
 				\argmin{D\in G}\;\frac{\lambda}{2}\left\|D - U + P\right\|_F^2 +
 				\beta\sum\limits_{i=1}^M\sum\limits_{j=1}^N d_{ij} \quad\text{s.t.}\quad
 				\sum\limits_{j=1}^N d_{ij}\leq v\,\forall i$;	\Comment{see Appendix
 													\ref{sec:positive_ell1_projection}}
 				\State $Z = \left(A^T A + \mu I\right)^{-1} \left(A^TW +
 				\mu\left(UB^T + Q\right)\right)$; 
 				\Statex								\Comment{Update Residuals}
				\State $S = \lambda\left(D^{\text{old}}-D\right) +
				\mu\left(Z^{\text{old}}-Z\right)B$; 
				\State $R_1 = \lambda\left(D-U\right)$;
				\State $R_2 = \mu\left(Z-UB^T\right)$;
				\Statex								\Comment{Lagrange Updates}
				\State $P = P - \left(D-U\right)$;
				\State $Q = Q -\left(Z-UB^T\right)$;
	\algstore{bkbreak}
	\end{algorithmic}
\end{algorithm}
	~\\
\begin{algorithm}[htp]
	\begin{algorithmic}[1]
	\algrestore{bkbreak}
				\Statex								\Comment{Varying Penalty/Lagrange Parameters}
				\If{$\lVert R_1\rVert_F>\eta_1\lVert S\rVert_F$}
 					\State $\lambda = \lambda\tau_1^{\text{incr}}$;
 					\State $P = \frac{P}{\tau_1^{\text{incr}}}$;
				\ElsIf{$\lVert S\rVert_F>\eta_1\lVert R_1\rVert_F$}
 					\State $\lambda = \frac{\lambda}{\tau_1^{\text{decr}}}$;
 					\State $P = P\tau_1^{\text{decr}}$;
				\EndIf
				\If{$\lVert R_2\rVert_F>\eta_2\lVert S\rVert_F$}
 					\State $\mu = \mu\tau_2^{\text{incr}}$;
 					\State $Q = \frac{Q}{\tau_2^{\text{incr}}}$;
					\ElsIf{$\lVert S\rVert_F>\eta_2\lVert R_2\rVert_F$}
 					\State $\mu = \frac{\mu}{\tau_2^{\text{decr}}}$;
 					\State $Q = Q\tau_2^{\text{decr}}$;
				\EndIf
				\Statex								\Comment{Stopping Criteria}
				\State  $\varepsilon_1^{\text{pri}} =
				\sqrt{MN}\;\varepsilon^{\text{abs}} +
				\varepsilon^{\text{rel}}\max \left\{\lVert U^k\rVert_F,\lVert D^k\rVert_F,0\right\}$;
				\State	$\varepsilon_2^{\text{pri}} =
				\sqrt{MT}\;\varepsilon^{\text{abs}} +
				\varepsilon^{\text{rel}}\max \left\{\lVert U^kB^T\rVert_F,\lVert Z^k\rVert_F,0\right\}$;
				\State  $\varepsilon^{\text{dual}} =
				\sqrt{MN}\;\varepsilon^{\text{abs}} +
				\varepsilon^{\text{rel}}\lVert\lambda P^k + \mu Q^k B\rVert_F$;
 			\EndWhile
			\State \textbf{return} $U$		\Comment{Solution of \eqref{eq:prob_for_algorithm}}
	\end{algorithmic}
\end{algorithm}

\section{Computational Experiments}
\label{sec:numerical_results}

In the last sections we have analyzed $\ell^{1,\infty}$-regularized
variational models and its reformulations. Moreover, we have deduced an
algorithm for the computation of a solution for the
$\ell^{1,\infty}$-regularized minimization problem.\\
In this section we propose dynamic positron emission tomography for the
visualization of myocardial perfusion as a possible application. To incorporate
knowledge about this application, we include \emph{kinetic modeling} in order to
model the blood flow and tracer exchange in the heart muscle.
After an introduction to the corresponding medical and mathematical background,
we show some results for synthetic examples and discuss the quality of our approach.

\subsection{Application to Dynamic Positron Emission Tomography}
\label{sec:applications}

Positron Emission Tomography (PET) is an imaging technique used in nuclear
medicine that visualizes the distribution of a radioactive tracer, which was
applied to the patient. Compared to computer tomography (CT), PET has the
advantage of being a functional rather than a morphological imaging technique. 

By using radioactive water ($H_2^{15}O$) as a tracer, it is possible to
visualize blood flow. $H_2^{15}O$ has the advantage of being highly diffusible and the
radiation exposure is low. Even dynamic images are possible. On the other hand
the reconstructed images have poor quality due to the short radioactive
half-life of $H_2^{15}O$.

Now let us consider the inverse problem of dynamic PET, i.e.
\begin{align}
	\mathcal{A}Z = W \; ,
	\label{eq:inv_prob}
\end{align}
where the operator $\mathcal{A}$ linking the dynamic image $Z$ with the measured
data $W$ is usually the \emph{Radon operator}, but could also be another
operator depending on the application.

By using kinetic modeling (cf. \cite[Chapter 23, p. 499 et seqq.]{Wernick2004})
we are able to describe the unknown image $Z$ as the tracer concentration in the
tissue $C_T$, i.e.
\begin{align}
	C_T(x,t) = F(x)\int\limits_0^t \!
	C_A(\tau)e^{-\frac{F(x)}{\lambda}(t-\tau)} \, \mathrm{d}\tau\; , 
	\label{eq:tracer_conc_tissue}
\end{align}
where $C_A(\tau)$ is the arterial tracer concentration, also known as input
curve, $F(x)$ refers to the perfusion and $\lambda$ is the ratio between the
tracer concentration in tissue and the venous tracer concentration resulting
from Fick's principle.
\begin{figure}[htb]
	\begin{center}
		\definecolor{wwccff}{rgb}{0.0,0.4,0.6}
\definecolor{ccwwqq}{rgb}{0.8,0.4,0.0}
\begin{tikzpicture}[scale=0.7,>=latex]
	\clip (-5.0,-0.15) rectangle (5.0,3.15);
	\shade [inner color=ccwwqq!80!white, outer color=ccwwqq!30!white
	] (-5.0,0.0) rectangle (5.0,1.5);
	\shadedraw [top color=wwccff!15!white, bottom color=wwccff!60!white, draw =
	black, line width=1.2pt] (-3.0,3.0) rectangle (3.0,1.5);
	\draw [line width=1.2pt] (-5.0,0.0) -- (5.0,0.0);
	\draw [line width=1.2pt] (-5.0,1.5) -- (5.0,1.5);
	\draw [->,color=red!50!black, line width=1.6pt] (-2.5,0.7) -- (2.5,0.7);
	\draw [->,color=mygreen, line width=1.6pt] (1.3,0.9) -- (1.3,2.7);
	\node at (-0.7,2.25) {Tissue $C_T$};
	\node at (0.0,1.1) {Blood};
	\node at (-3.5,0.7) {$C_A$};
	\node at (3.5,0.7) {$C_V$};
	\node at (0.0,0.35) {$F$};
	\node at (1.7,1.9) {$J_T$};
\end{tikzpicture}
		\caption{Illustration of kinetic modeling}
		\label{fig:Capillary_microcirculation}
	\end{center}
\end{figure}
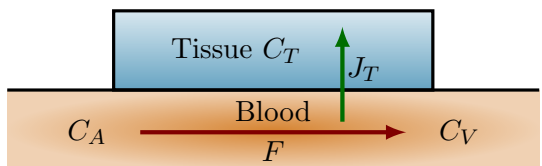

Kinetic modeling describes the tracer exchange with the tissue in the
capillaries. The tracer is injected and flows from the arteries with
concentration $C_A$ to the veins with concentration $C_V$.
While passing the capillaries between arteries and veins, a part of it
moves across the vascular wall with flux $J_T$ into the tissue, cf. Figure
\ref{fig:Capillary_microcirculation}.

Expression \eqref{eq:tracer_conc_tissue} is an integral equation including the
exponential factor $e^{-\frac{F(x)}{\lambda}(t-\tau)}$, which depends on
both input arguments, i.e. time $t$ and space $x$. 
This expression is highly nonlinear and thus not easy to handle especially
in combination with inverse problems. 
Due to the fact that we have prior knowledge about $\frac{F(x)}{\lambda}$,
i.e. that its value lies within certain parameters, we are
able to provide a big pool of given perfusion values for this expression,
which we denote by $\tilde{b}_j$.
Subsequently, we are able to consider a linearization, i.e.  
	\begin{equation}
		\mathcal{B}(u,C_A):=\sum\limits_{j=1}^N u_j(x)\underbrace{\int\limits_0^t \!
		C_A(\tau)e^{-\tilde{b}_j(t-\tau)} \, \mathrm{d}\tau}_{b_j(t)}\; , 
	\label{eq:kinetic_modeling_operator}
	\end{equation}
where $u_j(x)$ shall denote an approximation to the perfusion value
$F(x)$ corresponding to $\tilde{b}_j$.
Note that the integral is now independent of space.
Expression \eqref{eq:kinetic_modeling_operator} is reasonable if there is at
most one $u_j\neq 0$ for $j\in\left\{1,...,N\right\}$, i.e. the coefficient
$u_j$ corresponding to the correct perfusion value $\tilde{b}_j$.
In order to further simplify the work with this operator, we assume that the
input curve $C_A$ is predetermined.

Hereby we obtain the linear kinetic modeling operator
\begin{equation}
	\mathcal{B}(u)=\sum\limits_{j=1}^N u_j(x)b_j(t)\; ,
\label{eq:linear_operator}
\end{equation}
which we use to describe the unknown image $Z$.
The advantage of \eqref{eq:linear_operator} over \eqref{eq:tracer_conc_tissue}
is that we are able to compute the basis functions $b_j(t)$ in advance and thus
we can provide many of those for the reconstruction.\\
Note that there exists another deduction of \eqref{eq:linear_operator} by
\cite{AJReader2007}.

By considering a discretization of \eqref{eq:linear_operator},
we obtain
\begin{align}
	\left(UB^T\right)_{ik} = \sum\limits_{j=1}^N u_{ij}b_{kj}\; ,
	\label{eq:discretized_operator}
\end{align}
where $i\in\{1,\ldots,M\}$ denotes the pixel and
$k\in\{1,\ldots,T\}$ the time step.
After discretizing $\mathcal{A}$ as well, we can insert \eqref{eq:discretized_operator} for the
image $Z$ in \eqref{eq:inv_prob} and obtain
\begin{align}
	AUB^T = W \; .
	\label{eq:new_inv_prob}
\end{align}
Hence \eqref{eq:Uv_unconstr_prob} can be used for the
reconstruction of the discretized coefficients $u_{ij}$.\\

\subsection{Results}

In this section we present some numerical results.
We are going to work on synthetic data to investigate the
effectiveness of the approach.
In order to do so, we use a simple 3D matrix $\widehat{U}$ containing the exact
coefficients as ground truth, i.e. two spatial dimensions $M := m_1m_2$
and one extra dimension referring to the number of basis vectors $N$.
Defining two regions, where for only one basis vector the coefficients are
nonzero, yields the fact that the corresponding coefficients for most of the
basis vectors are zero.
Obviously our ground truth fulfills the prior knowledge, which we would like
to promote in the reconstruction, i.e. there is only one coefficient per pixel,
which is unequal to zero.

\begin{figure}[htbp]
	\begin{center}
  		\includegraphics[width=1\textwidth]{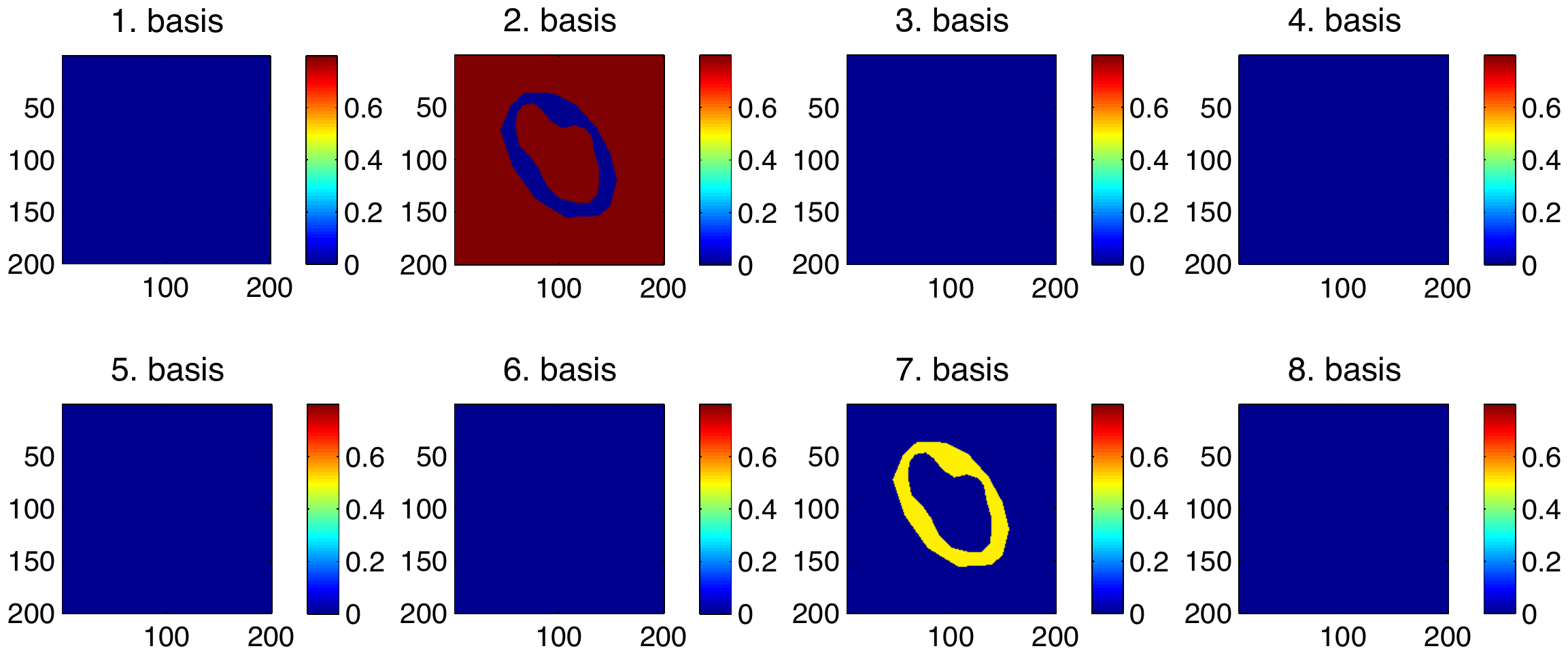}
	  \caption{Ground truth $\widehat{U}\in\R{200\times 200\times 8}$ with $200^2$
	  pixels and $8$ basis vectors}
  	  \label{fig:ground_truth}
	\end{center}
\end{figure}
In Figure \ref{fig:ground_truth} we see that the exact coefficients for the most
basis vectors are zero. Only some coefficients corresponding to the second and seventh
basis vectors are nonzero.
In order to obtain the artificial data $W\in\R{L\times T}$, we have to apply the
matrices $A\in\R{L\times M}$ and $B^T\in\R{N\times T}$ to the ground truth
$\widehat{U}\in\R{M\times N}$.\\
 
\begin{figure}[htbp]
	\begin{center}
	  	\includegraphics[width=0.6\textwidth]{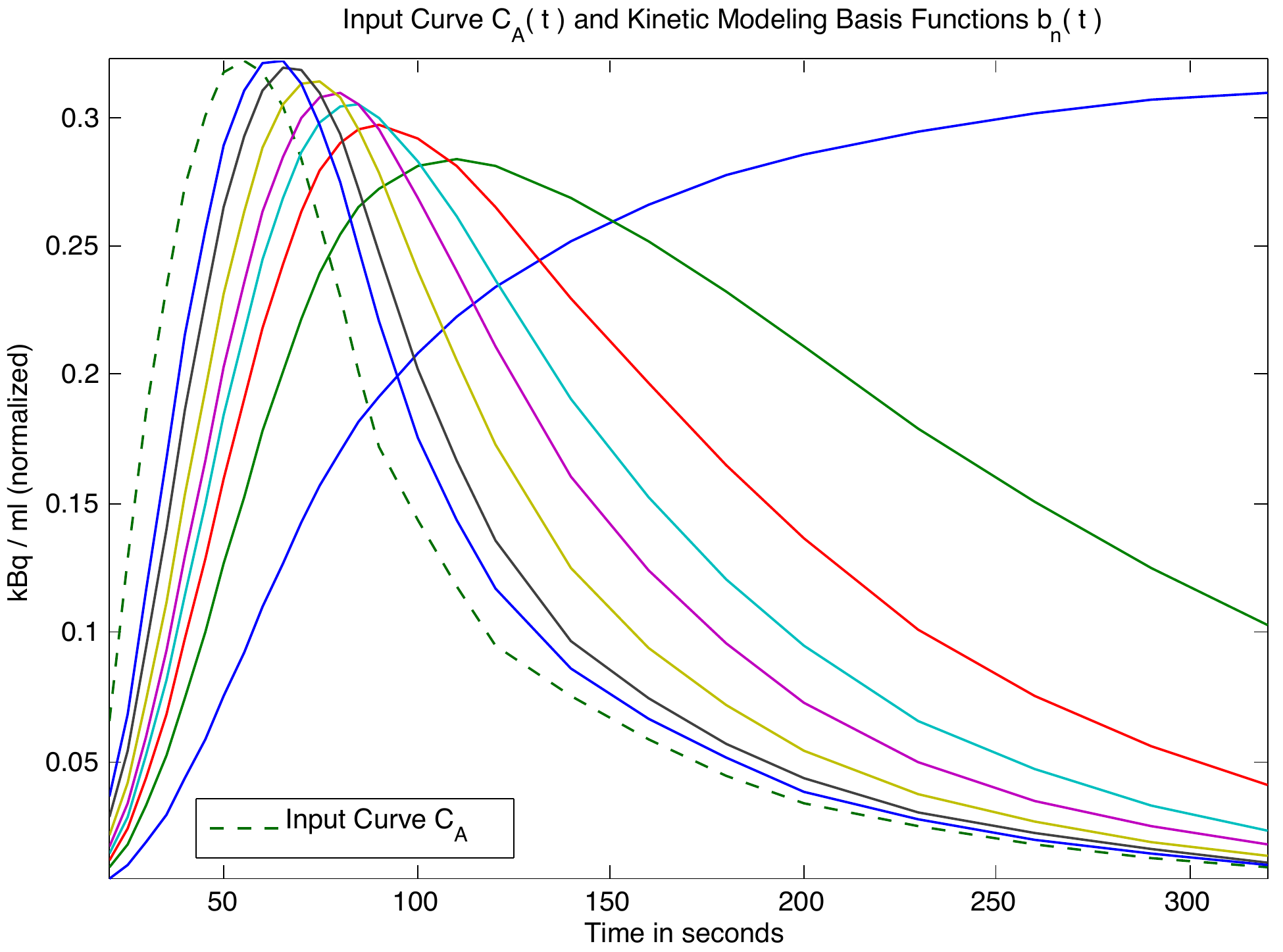}
	  \caption{Kinetic modeling basis vectors $B^T$}
	  \label{fig:dPET-Basis}
	\end{center}
\end{figure}

As an example for $B$ we use kinetic modeling basis vectors as they are used in
dynamic positron emission tomography (cf. Section \ref{sec:applications} and
\cite[Chapter 23]{Wernick2004}), which are basically discretized exponential
functions with different parameters.
In Figure \ref{fig:dPET-Basis} we observe that those basis vectors are very
similar, i.e. the mutual incoherence parameter (cf.
\eqref{eq:coherence_parameter}) is large.
For the verification of our approach including local sparsity, we use a simple
2D convolution in space for the matrix $A$ as a simplification.
In future work, however, the Radon operator shall be used instead.

By using Algorithm \ref{alg:L2_ell1_infty_via_ALM} on the so computed
data $W$ including a strong $\ell^{1,\infty}$-regularization, i.e.
$\tilde{v}=0.1$, we obtain a very good reconstruction of the support.

\begin{figure}[htbp]
	\begin{center}
	  	\includegraphics[width=1\textwidth]{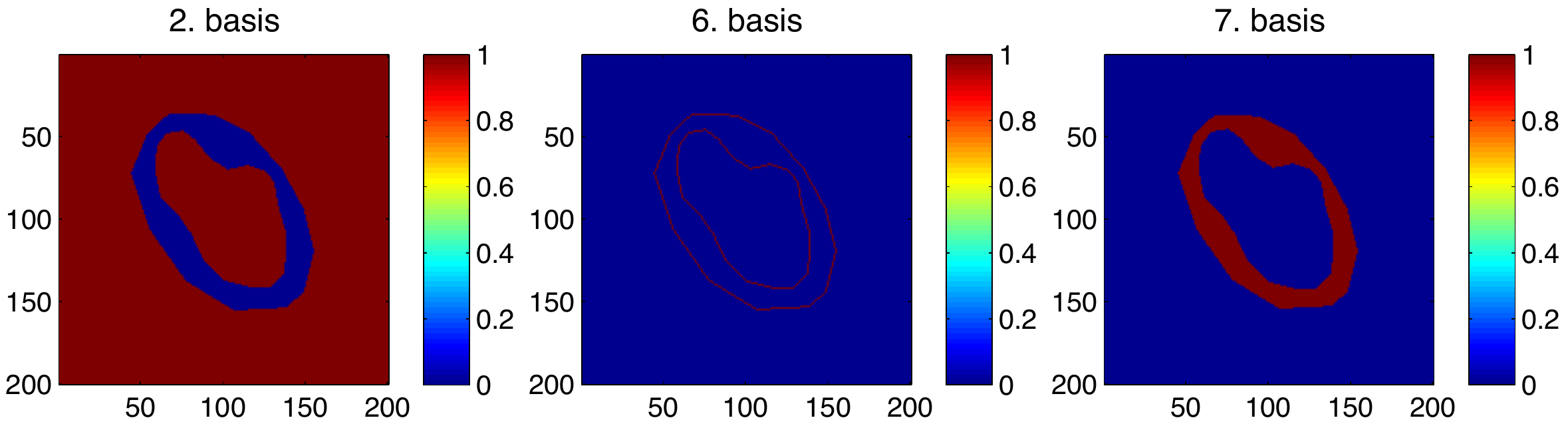} 
	  \caption{2nd, 6th and 7th reconstructed coefficient matrices using
	  $\tilde{v}=0.01$ and $\beta=0.1$}
	  \label{fig:Usupp_2_6_7_basis_heart}
	\end{center}
\end{figure}
Figure \ref{fig:Usupp_2_6_7_basis_heart} only shows the coefficient matrices to
those basis vectors, which include reconstructed nonzero coefficients. For
simplicity we do not show the other reconstructed coefficient matrices, which
are completely zero.
Obviously we obtain a very good reconstruction of the support.
Only a few coefficients, which actually correspond to the seventh basis vector
(dark brown), are reconstructed wrongly and show up in the sixth basis
vector (light green). This is due to the coherence of the basis vectors,
i.e. the sixth (light green) and the seventh (dark brown) basis
vector are very similar, compare for instance Figure \ref{fig:dPET-Basis}.

We observe that every value larger than $\tilde{v}$ is projected down to
$\tilde{v}$ and we make a systematic error.
This is due to the inequality constraint in problem
\eqref{eq:prob_for_algorithm} and because of the fact that we chose $\tilde{v}$
smaller than the maximal value of the exact data $\widehat{U}$ (compare for
instance Subsection \ref{equivalence_of_formulations} and especially Theorem
\ref{th:problem_implementation}). Thus we are not really close to the exact
data.
In order to overcome this problem, we first reconstruct the support including
the $\ell^{1,\infty}$- and $\ell^{1,1}$-regularization and then perform a second run
without regularization \emph{only on the known support} to reduce the distance
to the exact data.

\begin{figure}[htbp]
	\begin{center}
	  	\includegraphics[width=1\textwidth]{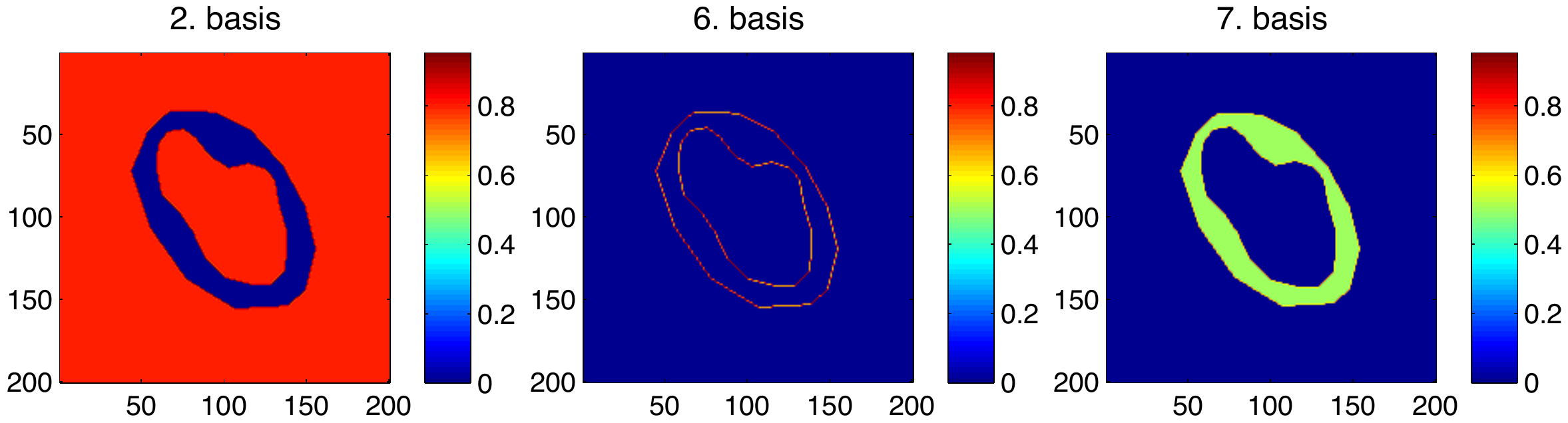} 
	  \caption{2nd, 6th and 7th reconstructed coefficient matrix using
	  $\tilde{v}=0.01$ and $\beta=0.1$ including a second run only on the
	  support; the other coefficient matrices are completely zero}
	  \label{fig:U_2_6_7_basis_heart}
	\end{center}
\end{figure}
In Figure \ref{fig:U_2_6_7_basis_heart} we see that this approach leads to very
good results.

We additionally reconstructed an example including some Gaussian noise.
In Figure \ref{fig:U_exactdata7_A1_B1_v001_eps01_sigma001_2ndRun1_kern16} we
observe that the algorithm performs quite nicely.\\

\begin{figure}[htbp]
	\begin{center}
	  	\includegraphics[width=1\textwidth]{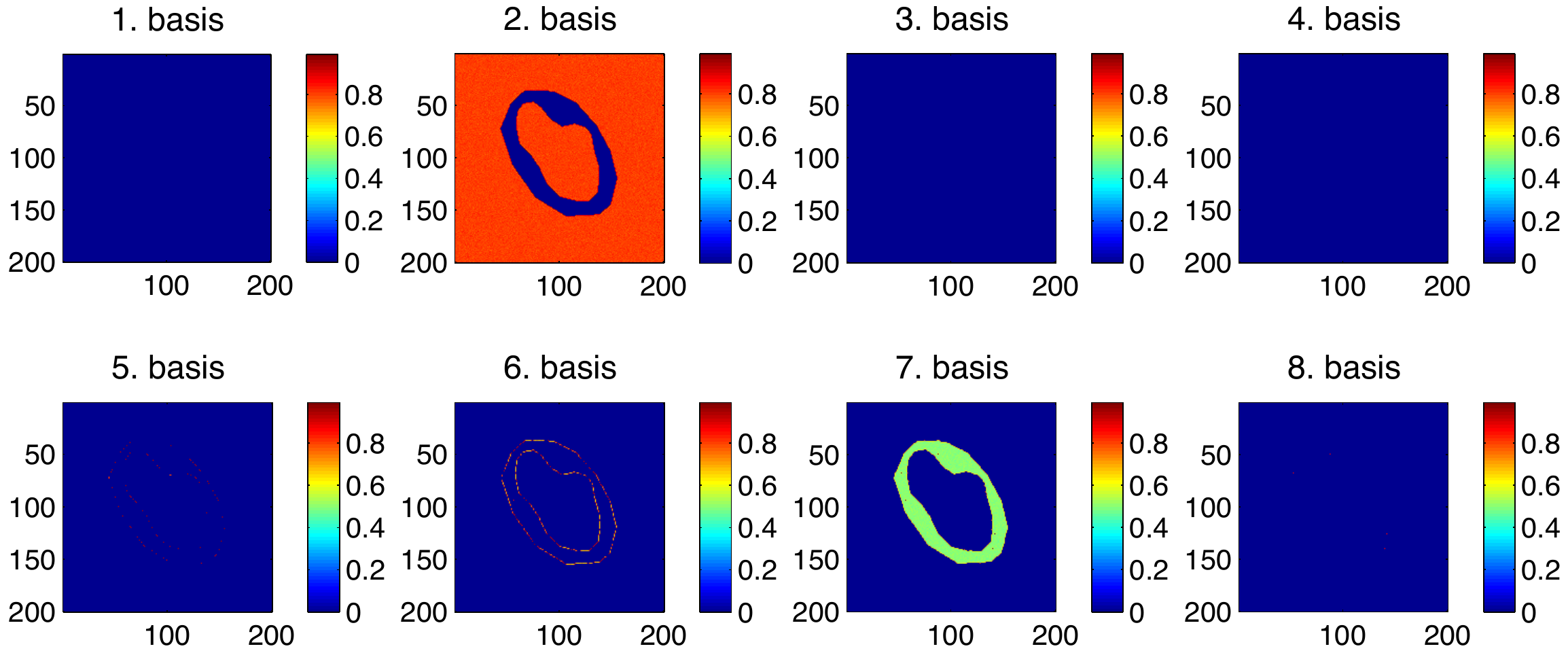}
	  	\caption{Reconstruction using $\tilde{v}=0.01$ and $\beta=0.1$ including
	  	Gaussian noise with standard deviation $\sigma = 0.01$}
	  \label{fig:U_exactdata7_A1_B1_v001_eps01_sigma001_2ndRun1_kern16}
	\end{center}
\end{figure}

Let us now evaluate Algorithm \ref{alg:L2_ell1_infty_via_ALM} with respect to
the quality of the reconstructed support.
In order to do so, we compare the reconstructed support after the first run
(including both regularizations) with the support of our ground truth and state
how much percent of the true support is reconstructed wrongly depending on the
$\ell^{1,\infty}$-regularization parameter $\tilde{v}$. We also include the
distance of the wrongly picked basis vector in each pixel, for instance if the
support of the ground truth picks basis vector number 7 and the reconstructed
support picks basis vector number 5 instead, we double the influence of the
error in this pixel if the reconstructed support picks basis vector number 4
instead of the correct number 7 we triple it and so on.

\begin{table*}[htp]
	\centering
		\begin{tabular}{| c | c | c | }
		  \hline
						& Percentage of & Number of \\
			$\tilde{v}$	& wrong pixel	& iterations \\
		  \hline
			$10^{-1}$ & 0.6722 \% & 262 \\
			$10^{-2}$ & 0.1772 \% & 350 \\
			$10^{-3}$ & 0.1772 \% & 431 \\
			$10^{-4}$ & 0.1772 \% & 512 \\
			$10^{-5}$ & 0.1772 \% & 592 \\
			$10^{-6}$ & 0.1772 \% & 663 \\
			$10^{-7}$ & 0.1772 \% & 672 \\
		  \hline
		\end{tabular}
	\caption{Evaluation of Algorithm \ref{alg:L2_ell1_infty_via_ALM} with $\beta =
	0.1$, $\lambda = 0.5$ and $\mu = 0.1$ }
	\label{tab:evaluation_no_noise}
\end{table*}
In Table \ref{tab:evaluation_no_noise} we see the evaluation of Algorithm
\ref{alg:L2_ell1_infty_via_ALM} applied to the noiseless data $W$.
When $\tilde{v}$ becomes smaller than $0.01$ we observe that there is no further
improvement.
As we have seen in Figure \ref{fig:Usupp_2_6_7_basis_heart} the boundary of the
region is reconstructed wrongly and the algorithm selects the sixth instead of
the seventh basis function.
However, the prior knowledge is already fulfilled, i.e. in every pixel there is
only one basis vector active.
This is the reason why there are still $0.1772$\% wrongly reconstructed
pixel and we do not obtain further improvement cannot be
achieved.

\begin{table*}[htp]
	\centering
		\begin{tabular}{| c | c | c | }
		  \hline
						& Percentage of & Number of \\
			$\tilde{v}$	& wrong pixel	& iterations \\
		  \hline
			$10^{-1}$ & 3.6997 \% & 262 \\
			$10^{-2}$ & 0.2628 \% & 350 \\
			$10^{-3}$ & 0.1991 \% & 431 \\
			$10^{-4}$ & 0.1875 \% & 512 \\
			$10^{-5}$ & 0.1897 \% & 592 \\
			$10^{-6}$ & 0.1925 \% & 663 \\
			$10^{-7}$ & 0.1859 \% & 672 \\
		  \hline
		\end{tabular}
	\caption{Evaluation of Algorithm \ref{alg:L2_ell1_infty_via_ALM} with $\beta =
	0.1$, $\lambda = 0.5$ and $\mu = 0.1$ including Gaussian noise with standard
	deviation $\sigma = 0.01$ }
	\label{tab:evaluation_noise_0.01}
\end{table*}
In Table \ref{tab:evaluation_noise_0.01} and \ref{tab:evaluation_noise_0.05} we
have the same error measures for different values of $\tilde{v}$ as in
Table \ref{tab:evaluation_no_noise}. However, this time we
included Gaussian noise on the data $W$ with standard deviation $0.01$,
$0.05$ respectively.
At first the error drops quickly. However, when $\tilde{v}$ becomes smaller the
error stagnates in a certain range similar to the noise-free case.

\begin{table*}[htp]
	\centering
		\begin{tabular}{| c | c | c | }
		  \hline
						& Percentage of & Number of \\
			$\tilde{v}$	& wrong pixel	& iterations \\
		  \hline
			$10^{-1}$ & 6.4534 \% & 262 \\
			$10^{-2}$ & 1.5500 \% & 350 \\
			$10^{-3}$ & 1.0353 \% & 431 \\
			$10^{-4}$ & 1.0084 \% & 512 \\
			$10^{-5}$ & 0.9725 \% & 592 \\
			$10^{-6}$ & 1.0044 \% & 663 \\
			$10^{-7}$ & 0.9625 \% & 672 \\
		  \hline
		\end{tabular}
	\caption{Evaluation of Algorithm \ref{alg:L2_ell1_infty_via_ALM} with $\beta =
	0.1$, $\lambda = 0.5$ and $\mu = 0.1$ including Gaussian noise with standard
	deviation $\sigma = 0.05$ }
	\label{tab:evaluation_noise_0.05}
\end{table*}
In order to smartly choose $\tilde{v}$, we have to find a good tradeoff between
a small error and a small number of iterations.
Choosing $\tilde{v}\in\left[10^{-4},\ldots,10^{-3}\right]$ seems to be a good choice.

\section{Conclusions}

For the solution of inverse problems, where the unknown is considered to be a
matrix, mixed $\ell^{p,q}$-norms can be used as regularization functionals in
order to promote certain structures in the reconstructed matrix.
Motivated by dynamic positron emission tomography for myocardial
perfusion, we proposed a novel variational model for a dictionary based matrix
completion problem incorporating local sparsity via
$\ell^{1,\infty}$-regularization as an alternative to the more commonly
considered joint sparsity model \cite{Fornasier2008, Teschke2007}.
We not only analyzed the existence and potential uniqueness of a solution, but
also investigated the subdifferential of the $\ell^{1,\infty}$-functional and a
source condition.
One of the main results of this paper consists of the deduction of an
equivalent formulation, which not only simplifies the analysis of the
problem, but also facilitates its numerical implementation.
Moreover, we discussed exact recovery for locally 1-sparse solutions by 
analyzing the noise-free case, in which we considered the minimization of the
nonnegative $\ell^{1,\infty}$-functional with an equality constraint in the
data fidelity term.
As a result of this analysis, we discovered that the dictionary matrix has to
be normalized in a certain way in order to exactly reconstruct locally 1-sparse
data under simplified conditions.\\
In this paper, a novel implementation of the problem was developed that relies
on a double splitting via the alternating direction method of multipliers (ADMM).
The algorithm yields superior results, in particular an almost exact recovery of
the true support of the solution.
Nevertheless, one drawback of the reformulation of the problem we introduced is
that the results are not very close to the true solution.
However, having a good estimate of the support of the solution allows us to
refine our first result by solving the inverse problem restricted to the
previously recovered support with no further regularization. 
This second result shows promising features, even in the presence of Gaussian
noise.\\
However, for some coefficients at the boundary of the exact nonzero region the
algorithm still picked the wrong basis vector.
In order to overcome this problem and to further improve the results, one can
add a total variation term to the variational regularization scheme, as we did
in \cite{Heins2014}.
Due to the additional regularization, which was not discussed in this
paper, the choice of a good combination of regularization parameters is
challenging, however, the results are promising.
By including total variation on the coefficient matrices, the reconstructions
improve even more and better results can be recovered.\\
In summary, the results obtained by our approach even without using total
variation are very satisfactory and could be even improved by incorporating an
additional total variation regularization on the coefficient matrices.
Our results motivate to investigate the model further, especially in
combination with a total variation regularization, which moreover makes the
algorithm more robust to noise.
Further research on parameter choice rules may eventually turn the approach
including total variation into an effective reconstruction scheme for practical
applications.

\bibliographystyle{plain}
\bibliography{bibliography}

\appendix

\section{Exact \texorpdfstring{$\ell^1$}{l1}-Reconstruction of 1-Sparse
Signals in 1D}
\label{sec:exact_reconstruction_of_a_delta_peak_in_1D}

In order to gain some understanding into suitable and necessary scaling
conditions recovering locally $1$-sparse solutions, we first consider the
simplest case, namely $M=1$, when the problem reduces to standard
$\ell^1$-minimization:

\begin{thm}[Exact Reconstruction of a 1-sparse Signal in 1D]
\label{th:exact_reconstruction_in_1D}
	~\\
	Let the vector $w:=e_j^TB^T=b_j^T$ be the $j$th basis vector
	and let $c=1-\left(\alpha+\beta\right)$ hold for
	$\left(\alpha+\beta\right)\in(0,1)$.\\
	If $\hat{u}=ce_j^T$ is the solution of \eqref{eq:varmodel1},
	then the matrix $B$ has to meet the scaling condition
	\begin{align}
		\lVert \gamma b_n\rVert_{\ell^2}=1\qquad\text{and}\qquad\left|\left\langle
		\gamma b_n,\gamma b_m\right\rangle\right|\leq 1 \quad\text{for}\quad n\neq
		m\; .
	\tag{\ref{eq:1D_peak_L2_exact_recovery}}
	\end{align}
\end{thm}
\begin{proof}
	~\\
	We firstly calculate the optimality condition of \eqref{eq:varmodel1} as
	\begin{align*}
		0=\left(\gamma^2\left(uB^T-w\right)B\right)_n+\left(\alpha+\beta\right)
		p_n\qquad\mathrm{with}\qquad p_n\in\partial|u_n| \; . 
	\end{align*}
	Then it follows that
	\begin{align*}
		p_n=\frac{1}{\alpha+\beta}\left(\gamma^2\left(w-uB^T\right)B\right)_n
	\end{align*}
	holds.
	Subsequently, we insert $\hat{u} = ce_j^T$ and $w = e_j^TB^T$ to obtain
	\begin{align*}
		p_n
		&=\frac{1}{\alpha+\beta}\left(\gamma^2\left(e_j^TB^T-ce_j^TB^T\right)B\right)_n\\
				&=\frac{1-c}{\alpha+\beta}\left(\gamma^2 e_j^TB^TB\right)_n \\
				&=\frac{1-c}{\alpha+\beta}\left(\gamma^2 b_j^TB\right)_n \\
				&=\frac{1-c}{\alpha+\beta}\sum\limits_{k=1}^T\gamma_t^2b_{tj}b_{tn} \\
				&=\frac{1-c}{\alpha+\beta}\left\langle\gamma b_j,\gamma b_n\right\rangle
	\end{align*}
	for every $n\in\{1,\ldots,N\}$.
	Since $p_n\in\partial\left|\hat{u}_n\right|$ has to be satisfied for
	all $n\in\{1,\ldots,N\}$, we need to ensure that
	\begin{align*}
		p_j=1\quad\mathrm{and}\quad p_i\in\left[-1,1\right]\quad\mathrm{for}\quad i\neq j
	\end{align*}
	hold, which is true under the assumptions mentioned above.
\end{proof}

For this reason we know that we have to normalize our basis vectors with respect
to the $\ell^2$-norm to reconstruct at least a $\delta$-peak exactly in one dimension. 
Note that in the one-dimensional case the $\ell^{1,\infty}$-regularization and
the $\ell^{1,1}$-regularization reduce to a single $\ell^1$-regularization with
regularization parameter $\alpha + \beta$.

We further analyze the special case of the
\emph{Kullback-Leibler approximation} (cp. \cite[pp. 58-59]{Sawatzky2011}).

\begin{thm}[Exact Recovery of a $\delta$-Peak in 1D with KL-Approximation]
\label{th:exact_reconstruction_in_1D_KL_approx}
	~\\
	Let the vector $w:=e_j^T B^T=b_j^T$ be the $j$th basis
	vector and let $c=1-\left(\alpha+\beta\right)$ hold for
	$c\in(0,1)$.\\
	In case that $\hat{u}=ce_j^T$ is the solution of
	\eqref{eq:varmodel1} with $\gamma=\frac{1}{\sqrt{w}}$, then the columns of the matrix $B$ have to
	be normalized in the $\ell^1$-norm, i.e.
	\begin{align*}
		\lVert b_n\rVert_{\ell^1}=1\qquad\forall\; n\in\left\{1,...N\right\}\; . 
	\end{align*}
\end{thm}
\begin{proof}
	~\\
	We first compute the optimality condition of \eqref{eq:varmodel1}
	with $\gamma=\frac{1}{\sqrt{w}}$ as 
	\begin{align}
		0=\left(\left(\frac{1}{\sqrt{w}}\left(\frac{1}{\sqrt{w}}\left(uB^T-w\right)\right)\right)B\right)_n+\left(\alpha+\beta\right)
		p_n\qquad\mathrm{with}\qquad p_n\in\partial\left|u_n\right|\; .  
		\label{eq:opt_cond_KL_approx}
	\end{align}
	It follows that
	\begin{align*}
		p_n&=\frac{1}{\alpha+\beta}\left(\left(\frac{1}{w}\left(w-uB^T\right)\right)B\right)_n\\
		&=\frac{1}{\alpha+\beta}\left(\left(\mathds{1}_t^T-\frac{1}{w}uB^T\right)B\right)_n
	\end{align*}
	holds.
	Then we insert $\hat{u}$ and $w$ and conclude
	\begin{align*}
		p_n	&= \frac{1}{\alpha+\beta}
		\left(\left(\mathds{1}_t^T-\frac{c e_j^T B^T}{e_j^T B^T}\right)B\right)_n\\
		&= \frac{1-c}{\alpha+\beta} \left(\mathds{1}_t^TB\right)_n\\
		&= \frac{1-c}{\alpha+\beta} \sum\limits_{k=1}^T b_{tn}\\
		&= \frac{1-c}{\alpha+\beta} \lVert b_n\rVert_{\ell1}\; .
	\end{align*}
	With the assumptions mentioned above we obtain again
	$p_n\in\partial\left|\hat{u}_n\right|$.
\end{proof}

It is worth mentioning that in this case every positive solution of $Bu=w$ meets
the optimality condition \eqref{eq:opt_cond_KL_approx}, in particular every
non-sparse solution.

\section{Solving the Positive
\texorpdfstring{$\ell^{1,\infty}-\ell^{1,1}$}{ell(1,infty)-ell11}-Projection-Problem}
\label{sec:positive_ell1_projection}

We want to solve the following problem
\begin{align}
	\min\limits_{D\in G}\;\frac{\lambda}{2}\left\|D - U + P\right\|_F^2 +
	\beta\sum\limits_{i=1}^M\sum\limits_{j=1}^N d_{ij} \quad\text{s.t.}\quad
	\sum\limits_{j=1}^N d_{ij}\leq \tilde{v} \;\; \forall \,
	i\in\{1,\ldots,M\} .
	\label{eq:posell1proj-problem}
\end{align}
In order to do so, we reformulate the first part of the problem, i.e.
\begin{align*}
	&\frac{\lambda}{2}\left\|D-U+P\right\|^2_F +
	\beta\sum\limits_{i=1}^M\sum\limits_{j=1}^N d_{ij} \\
	= &\sum\limits_{i=1}^M\sum\limits_{j=1}^N
	\left(\frac{\lambda}{2}\left(d_{ij}-u_{ij}+p_{ij}\right)^2 +
	\beta d_{ij}\right) \\
	= &\sum\limits_{i=1}^M\sum\limits_{j=1}^N \frac{\lambda}{2}\left(d_{ij}^2-
	2d_{ij}\left(u_{ij}+p_{ij}\right) + \left(u_{ij}+p_{ij}\right)^2 +
	\frac{2\beta}{\lambda} d_{ij}\right) \\
	= &\sum\limits_{i=1}^M\sum\limits_{j=1}^N \frac{\lambda}{2}\left(d_{ij}^2 -
	2d_{ij}\left(u_{ij}+p_{ij}-\frac{\beta}{\lambda}\right) +
\left(u_{ij} + p_{ij}\right)^2 \right) \\
	= &\sum\limits_{i=1}^M\sum\limits_{j=1}^N
	\frac{\lambda}{2}\left(d_{ij}
	- \left(u_{ij}+p_{ij}-\frac{\beta}{\lambda}\right)\right)^2 -
	\frac{\lambda}{2}\left(\left(u_{ij}+p_{ij}
	- \frac{\beta}{\lambda}\right)^2 + \left(u_{ij}+p_{ij}\right)^2\right)\; .
\end{align*}
Since the last part of the sum is independent of $d_{ij}$, we can
consider
\begin{align*}
	\min\limits_{D\in G}\;\frac{\lambda}{2}\left\|D
	- U + P - \frac{\beta}{\lambda}\mathds{1}_{M\times N}\right\|_F^2
	\quad\text{s.t.}\quad \sum\limits_{j=1}^N d_{ij}\leq \tilde{v} \;\;
	\forall \, i\in\{1,\ldots,M\}
\end{align*}
instead. We minimize this with respect to every row independently, i.e.
\begin{align}
	\min\limits_{D\in\R{M\times N}}\;\frac{\lambda}{2}\left\|d_{(i)} - u_{(i)} +
	p_{(i)} - \frac{\beta}{\lambda}\mathds{1}_{N}\right\|_2^2
	\label{eq:subprob_in_D_rows}
\end{align}
where  $i\in\{1,\ldots,M\}$ holds and with respect to the
constraints
\begin{align}
	\left(d_{(i)}\right)_j\geq 0\quad\forall\,j\in\{1,\ldots,N\}\; ,
	\label{eq:constraint_nr_1} \tag{Constr1} \\
	\sum\limits_{j=1}^N 	\left(d_{(i)}\right)_j\leq \tilde{v} \; ,
	\label{eq:constraint_nr_2} \tag{Constr2}
\end{align}
with $d_{(i)}$ denoting the $i$th \emph{transposed} row of $D$,
for $u_{(i)}$, $p_{(i)}$ respectively.\\
In order to minimize this problem, we first consider
\eqref{eq:subprob_in_D_rows} only under \eqref{eq:constraint_nr_1}. In this case
the solution is given by
\begin{align}
	\widetilde{d}_{(i)}=\max\left\{u_{(i)}+p_{(i)} -
	\frac{\beta}{\lambda}\mathds{1}_{N},0\right\}\; .
\label{eq:sol_under_const1}
\end{align}
To include \eqref{eq:constraint_nr_2}, we have to do a
case-by-case-analysis:

~\\
\underline{Case a:}
	\begin{addmargin}[0.3cm]{0.3cm}
		Let \eqref{eq:sol_under_const1} satisfy \eqref{eq:constraint_nr_2}. In this
		case the solution of \eqref{eq:subprob_in_D_rows} under
		\eqref{eq:constraint_nr_1} and \eqref{eq:constraint_nr_2} is given by
		\begin{align*}
			d_{(i)}=\widetilde{d}_{(i)}\; .
		\end{align*}
	\end{addmargin}
	\underline{Case b:}
	\begin{addmargin}[0.3cm]{0.3cm}
			Let \eqref{eq:sol_under_const1} \emph{not} satisfy
			\eqref{eq:constraint_nr_2}, i.e.
			$\sum\limits_{j=1}^N\left(\widetilde{d}_{(i)}\right)_j>\tilde{v}$. Then the
			solution of \eqref{eq:subprob_in_D_rows} under \eqref{eq:constraint_nr_1} and
			\eqref{eq:constraint_nr_2} has to fulfill
			\begin{align*}
				\sum\limits_{j=1}^N\left(d_{(i)}\right)_j=\tilde{v} \tag{Constr3}\; .
			\label{eq:constraint_nr_2_star}
			\end{align*}
			Thus we have to solve \eqref{eq:subprob_in_D_rows} under
			\eqref{eq:constraint_nr_1} and \eqref{eq:constraint_nr_2_star}.
			For this purpose we propose the corresponding Lagrange functional as
			\begin{align}
				\begin{split}
					\mathcal{L}^{\lambda}(d_{(i)},\mu_{(i)},\vartheta) = 
					&\min\limits_{D\in\R{M\times N}}\;\frac{\lambda}{2}\left\|d_{(i)} - u_{(i)}
					+ p_{(i)} - \frac{\beta}{\lambda}\mathds{1}_{N}\right\|_2^2  \\
					&+\vartheta\left(\sum\limits_{j=1}^N\left(d_{(i)}\right)_j -
					\tilde{v}\right) - \sum\limits_{j=1}^N\left(d_{(i)}\right)_j\left(\mu_{(i)}\right)_j\; .
				\end{split}
			\label{eq:lagr_secsub}
			\end{align}
			Once we know $\vartheta$ we can compute the optimal $d_{(i)}$ as
			\begin{align}
				\begin{split}
					d_{(i)} 
					&= \mathrm{shrink}^{+}\left(u_{(i)} + p_{(i)} -
					\frac{\beta}{\lambda}\mathds{1}_{N}, \;
					\frac{\vartheta}{\lambda}\mathds{1}_{N}\right) \\
					:&= \max\left\{u_{(i)} + p_{(i)} - \frac{\beta}{\lambda}\mathds{1}_{N} -
					\frac{\vartheta}{\lambda}\mathds{1}_{N}, \;0\right\} \, .
				\end{split}
			\label{eq:sol_const2_star}
			\end{align}
			We can see this by computing the optimality condition of
			\eqref{eq:lagr_secsub}, i.e.
			\begin{align}
				\begin{split}
					0 &= \partial_{d_{(i)}}\mathcal{L}^{\lambda}(d_{(i)},\mu_{(i)},\vartheta) \\
					&= \lambda\left(d_{(i)} - u_{(i)} + p_{(i)} -
					\frac{\beta}{\lambda}\mathds{1}_{N}\right) + \vartheta\mathds{1}_{N} -
					\mu_{(i)}\; ,
				\end{split}
			\label{eq:opt_cond_lagr_secsub}
			\end{align}
			with the complementary conditions
			\begin{align*}
				\left(\mu_{(i)}\right)_j\geq
				0\qquad\mathrm{and}\qquad\left(\mu_{(i)}\right)_j\left(d_{(i)}\right)_j=0\;\;\forall
				j\in\{1,\ldots,N\}\; .
			\end{align*}
			If $\left(d_{(i)}\right)_j\neq 0$ holds, then we have
			$\left(\mu_{(i)}\right)_j=0$ and thus we obtain from \eqref{eq:opt_cond_lagr_secsub} that
			\begin{align*}
				\qquad 0 &= \lambda\left(d_{(i)} - u_{(i)} + p_{(i)} -
				\frac{\beta}{\lambda}\mathds{1}_{N}\right) + \vartheta\mathds{1}_{N} \\
				\Leftrightarrow\; d_{(i)} &= u_{(i)}
				- p_{(i)} + \frac{\beta-\vartheta}{\lambda}\mathds{1}_{N}
			\end{align*}
			has to be true.
			On the other hand, if $\left(d_{(i)}\right)_j=0$, then
			$\left(\mu_{(i)}\right)_j\geq 0$. Hence we see from \eqref{eq:opt_cond_lagr_secsub} that
			\begin{align*}
				\left(\mu_{(i)}\right)_j = \lambda\left(p_{(i)} -
				u_{(i)}\right)_j - \beta + \vartheta &\geq 0\; , \\
				\Leftrightarrow\qquad\qquad\quad
				 \left(u_{(i)} - p_{(i)}\right)_j +
				 \frac{\beta-\vartheta}{\lambda} &\leq 0\; .
			\end{align*}
			The Lagrange parameter $\vartheta$ should be chosen such that
			\eqref{eq:constraint_nr_2_star} holds. Therefore we investigate
			\begin{align*}
				\sum\limits_{j\in I}\left(\left(u_{(i)} - p_{(i)}\right)_j +
				 \frac{\beta-\vartheta}{\lambda} \right) = v\; ,
			\end{align*}
			with the set $I$ containing all indices, for which
			\begin{align}
				\left(u_{(i)} - p_{(i)}\right)_j +
				 \frac{\beta-\vartheta}{\lambda}\geq 0
			\label{eq:condition_for_set}
			\end{align}
			holds, since for all other indices $j\notin I$ the term $\;\left(u_{(i)} -
			p_{(i)}\right)_j + \frac{\beta-\vartheta}{\lambda}$ is projected to $0$ due to
			\eqref{eq:sol_const2_star}.\\
			Hence we obtain
			\begin{align*}
				\vartheta = \frac{\lambda}{\left|I\right|}\left(\sum\limits_{j\in
				I}\left(u_{(i)} - p_{(i)}\right)_j + \frac{\beta}{\lambda} -
				\tilde{v}\right)\; .
			\end{align*}
			Now we have to compute $I$.	Since we are able to sort the vectors according
			to value, it is sufficient to find $\left|I\right|$.
			Then we obtain
			\begin{align*}
				\vartheta =
				\frac{\lambda}{\left|I\right|}\left(\sum\limits_{r=1}^{\left|I\right|}\widehat{\left(u_{(i)}
				- p_{(i)}\right)}_r + \frac{\beta}{\lambda} - \tilde{v}\right)\; ,
			\end{align*}
			where $\;\widehat{\cdot}\;$ denotes the respective vector sorted according to value.\\
			In order to obtain $\left|I\right|$, we use the following\\
			\begin{thm}[{\cite[p. 3]{Duchi2008}}]
			~\\
				Let $\widehat{\left.u_{(i)} - p_{(i)}\right.}$ denote the vector obtained
				by sorting $u_{(i)} - p_{(i)}$ in a descending order.
				Then the number of indices, for which \eqref{eq:condition_for_set} holds, is
				\begin{align*}
					\left|I\right| = \max\left\{j:\lambda\widehat{\left(u_{(i)}
					- p_{(i)}\right)}_j + \beta -
					\frac{\lambda}{j}\left(\sum\limits_{r=1}^j\widehat{\left(u_{(i)}
					- p_{(i)}\right)}_r + \frac{\beta}{\lambda} - \tilde{v}\right) > 0\right\}
					\; .
				\end{align*}
			\label{th:index_nr_theorem}
			\end{thm}
	\end{addmargin}

Now we are able to propose the solving algorithm.

\begin{algorithm}
\caption{Positive
\texorpdfstring{$\ell^{1,\infty}$-$\ell^{1,1}$}{l(1,infty)-l11}-Projection}
\label{alg:positive_ell1infty_ell11_projection}
	\begin{algorithmic}[1]
			\State \textbf{Parameters:} $U\in\R{M\times	N},\;P\in\R{M\times
										N},\tilde{v}>0,\;\beta>0,\;\lambda>0,\;M,N\in\mathbb{N}$ 
			\State \textbf{Initialization:} $D\equiv 0,\;\left|I\right|=0,\;\vartheta=0$
			\ForAll{$i\in\left\{1,...,M\right\}$}
					\Statex
					\State $\widetilde{d}_{(i)}=\max\left\{u_{(i)}+p_{(i)} -
							\frac{\beta}{\lambda}\mathds{1}_{N},0\right\}$;
					\Statex

					\If {$\sum\limits_{j=1}^N \widetilde{d}_{ij}\leq \tilde{v}$} 
					\Comment{Solve with \eqref{eq:constraint_nr_1} and \eqref{eq:constraint_nr_2}}
						\Statex
						\State $d_{(i)} = \widetilde{d_{(i)}}$;
						\Else \Comment{Solve with \eqref{eq:constraint_nr_1} and \eqref{eq:constraint_nr_2_star}}
						\State $\left|I\right| = \max\left\{j:\lambda\widehat{\left(u_{(i)}
								- p_{(i)}\right)}_j + \beta -
								\frac{\lambda}{j}\left(\sum\limits_{r=1}^j\widehat{\left(u_{(i)}
								- p_{(i)}\right)}_r + \frac{\beta}{\lambda} - \tilde{v}\right) >
								0\right\}$;
						\Statex
						\State $\vartheta =
								\frac{\lambda}{\left|I\right|}\left(\sum\limits_{r=1}^{\left|I\right|}\widehat{\left(u_{(i)}
								- p_{(i)}\right)}_r + \frac{\beta}{\lambda} - \tilde{v}\right)$;
						\Statex
						\State $d_{(i)} = \mathrm{shrink}^{+}\left(u_{(i)} + p_{(i)} -
								\frac{\beta}{\lambda}\mathds{1}_{N},\;
								\frac{\vartheta}{\lambda}\mathds{1}_{N}\right)$;
					\EndIf
				\EndFor
			\State \textbf{return} $D$		\Comment{Solution of \eqref{eq:posell1proj-problem}}
	\end{algorithmic}
\end{algorithm}

\section{Inequality for Stopping Criteria}
\label{sec:inequality}

In order to proof the inequality
\begin{align}
	\begin{split}
		&\frac{1}{2}\left\| AZ-W\right\|_F^2 +
		\beta\sum\limits_{i=1}^M\sum\limits_{j=1}^N d_{ij} + J(D) - Y^* \\
		\leq\; &\left\langle P^k,R_1^k\right\rangle_F + \left\langle
		Q^k,R_2^k\right\rangle_F + \left\langle U^k - U^*,S^k\right\rangle_F \; ,
	\end{split}
	\tag{\ref{eq:inequality}}
\end{align}
which is needed in Subsection \ref{sec:stopping_criteria}, we are going to adapt
the proof of \cite[Appendix A]{Boyd2010} to the case of our double splitting.

Let us consider the unscaled augmented Lagrangian
\eqref{eq:unsc_aug_lagrange}.
By definition $U^{k+1}$ minimizes
$$\mathcal{L}_{un}^{\lambda,\mu}\left(U,D^k,Z^k;\widetilde{P}^k,\widetilde{Q}^k \right) \; ,$$
$D^{k+1}$ minimizes
$$\mathcal{L}_{un}^{\lambda,\mu}\left(U^{k+1},D,Z^k;\widetilde{P}^k,\widetilde{Q}^k
\right)$$
and $Z^{k+1}$ minimizes
$$\mathcal{L}_{un}^{\lambda,\mu}\left(U^{k+1},D^{k+1},Z;\widetilde{P}^k,\widetilde{Q}^k
\right) \; .$$
We now have to examine the optimality conditions.

~\\
\underline{OPT1}:
\begin{addmargin}[0.3cm]{0.3cm}
	By starting with
	\begin{align*}
		0 &\in
		\partial_U\mathcal{L}_{un}^{\lambda,\mu}\left(U^{k+1},D^k,Z^k;\widetilde{P}^k,\widetilde{Q}^k\right)
		\\
		&= \widetilde{P}^k + \lambda(U^{k+1}-D^k) + \widetilde{Q}^kB +
		\mu(U^{k+1}B^T-Z^k)B \; ,
	\end{align*}
	we insert the Lagrange updates
	\begin{align}
		\widetilde{P}^k = \widetilde{P}^{k+1} + \lambda(D^{k+1}-U^{k+1})
		\quad\text{and}\quad \widetilde{Q}^k = \widetilde{Q}^{k+1} +
		\mu(Z^{k+1}-U^{k+1}B^T)
	\label{eq:lagrange_updates}
	\end{align}
	and obtain
	\begin{align*}
		0 \in \widetilde{P}^{k+1} + \widetilde{Q}^{k+1}B + \lambda(D^{k+1}-D^k) +
		\mu(Z^{k+1}-Z^k)B \; .
	\end{align*}
	Thus we see that $U^{k+1}$ minimizes
	\begin{align*}
		\left\langle\widetilde{P}^{k+1}+\widetilde{Q}^{k+1}B,U\right\rangle_F +
		\lambda\left\langle D^{k+1}-D^k,U\right\rangle_F + \mu\left\langle
		Z^{k+1}-Z^k,UB^T\right\rangle_F \; .
	\end{align*}
\end{addmargin}

~\\
\underline{OPT2}:
\begin{addmargin}[0.3cm]{0.3cm}
	Here we have
	\begin{align*}
		0 &\in
		\partial_D\mathcal{L}_{un}^{\lambda,\mu}\left(U^{k+1},D^{k+1},Z^k;\widetilde{P}^k,\widetilde{Q}^k\right)
		\\
		&= \beta\mathds{1}_{M\times N} - \widetilde{P}^k - \lambda(U^{k+1}-D^{k+1}) +
		\partial J(D^{k+1}) \; ,
	\end{align*}
	with $J$ as defined in \eqref{eq:positive_ell1infty}. Inserting $\widetilde{P}^k$
	from \eqref{eq:lagrange_updates} yields
	\begin{align*}
		0 \in \beta\mathds{1}_{M\times N} - \widetilde{P}^{k+1}	+ \partial J(D^{k+1})
	\end{align*}
	and hence we see that $D^{k+1}$ minimizes
	\begin{align*}
		\beta\sum\limits_{i=1}^M\sum\limits_{j=1}^N d_{ij} -
		\left\langle\widetilde{P}^{k+1},D\right\rangle_F \quad\text{s.t.}\quad
		\sum\limits_{j=1}^N d_{ij}\leq \tilde{v}, \; d_{ij}\geq 0 \; .
	\end{align*}
\end{addmargin}
\underline{OPT3}:
\begin{addmargin}[0.3cm]{0.3cm}
	In this case we compute
	\begin{align*}
		0 &\in
		\partial_Z\mathcal{L}_{un}^{\lambda,\mu}\left(U^{k+1},D^{k+1},Z^{k+1};\widetilde{P}^k,\widetilde{Q}^k\right)
		\\
		&= A^T(AZ^{k+1}-W) - \widetilde{Q}^k - \mu(U^{k+1}B^T-Z^{k+1}) \; .
	\end{align*}
	Inserting $\widetilde{Q}^k$	from \eqref{eq:lagrange_updates} yields
	\begin{align*}
		0 \in A^T(AZ^{k+1}-W) - \widetilde{Q}^{k+1} \; .
	\end{align*}
	Therefore $Z^{k+1}$ minimizes
	\begin{align*}
		\frac{1}{2}\left\|AZ-W\right\|_F^2 -
		\left\langle\widetilde{Q}^{k+1},Z\right\rangle_F \; .
	\end{align*}
\end{addmargin}
All in all it follows that
\begin{align}
	\begin{split}
		&\;\left\langle\widetilde{P}^{k+1}+\widetilde{Q}^{k+1}B,U^{k+1}\right\rangle_F +
		\lambda\left\langle D^{k+1}-D^k,U^{k+1}\right\rangle_F + \mu\left\langle
		Z^{k+1}-Z^k,U^{k+1}B^T\right\rangle_F \\
		\leq
		&\;\left\langle\widetilde{P}^{k+1}+\widetilde{Q}^{k+1}B,U^*\right\rangle_F +
		\lambda\left\langle D^{k+1}-D^k,U^*\right\rangle_F + \mu\left\langle
		Z^{k+1}-Z^k,U^*B^T\right\rangle_F
	\end{split}
	\label{eq:unequality_U}
\end{align}
and
\begin{align}
	\beta\sum\limits_{i=1}^M\sum\limits_{j=1}^N d_{ij}^{k+1} -
	\left\langle\widetilde{P}^{k+1},D^{k+1}\right\rangle_F + J(D^{k+1})
	\;\leq\;
	\beta\sum\limits_{i=1}^M\sum\limits_{j=1}^N d^*_{ij} -
	\left\langle\widetilde{P}^{k+1},D^*\right\rangle_F + J(D^*)
	\label{eq:unequality_D}
\end{align}
and
\begin{align}
	\frac{1}{2}\left\|AZ^{k+1}-W\right\|_F^2 -
	\left\langle\widetilde{Q}^{k+1},Z^{k+1}\right\rangle_F
	\;\leq\;
	\frac{1}{2}\left\|AZ^*-W\right\|_F^2 -
	\left\langle\widetilde{Q}^{k+1},Z^*\right\rangle_F
	\label{eq:unequality_Z}
\end{align}
have to hold.
Adding equations \eqref{eq:unequality_U}, \eqref{eq:unequality_D} and
\eqref{eq:unequality_Z} together leads to
\begin{align*}
	&\frac{1}{2}\left\|AZ^{k+1}-W\right\|_F^2 +
	\beta\sum\limits_{i=1}^M\sum\limits_{j=1}^N d_{ij}^{k+1} + J(D^{k+1}) -
	\frac{1}{2}\left\|AZ^*-W\right\|_F^2 -
	\beta\sum\limits_{i=1}^M\sum\limits_{j=1}^N d^*_{ij} - J(D^*) \\
	\;\leq\;
	&\left\langle\widetilde{P}^{k+1},D^{k+1}-U^{k+1} \right\rangle_F +
	\left\langle\widetilde{Q}^{k+1},Z^{k+1}-U^{k+1}B^T \right\rangle_F +
	\lambda\left\langle D^{k+1}-D^k,U^*-U^{k+1} \right\rangle_F \\ 
	&+ \mu\left\langle Z^{k+1}-Z^k,(U^*-U^{k+1})B^T \right\rangle_F +
	\left\langle\widetilde{P}^{k+1},U^*-D^* \right\rangle_F +
	\left\langle\widetilde{Q}^{k+1},U^*B^T-Z^* \right\rangle_F
\end{align*}
By using the definitions of $R_{1,2}^{k+1}$ and $S^{k+1}$ (see for instance
\eqref{eq:primal_res1},\eqref{eq:primal_res2} and \eqref{eq:dual_res}) and the
fact that we have $U^* = D^*$ and $U^*B^T = Z^*$, we finally obtain
\begin{align*}
	\frac{1}{2}\left\|AZ^{k+1}-W\right\|_F^2 +
	\beta\sum\limits_{i=1}^M\sum\limits_{j=1}^N d_{ij}^{k+1} + J(D^{k+1}) -
	Y^* 
	\leq \left\langle P,R_1^k\right\rangle_F + \left\langle Q,R_2^k\right\rangle_F +
	\left\langle S^k,U^k-U^*\right\rangle_F \; .
\end{align*}

\end{document}